\documentclass[12pt,reqno]{amsart}

\usepackage{color}
\usepackage{amsmath, amsthm, amssymb}
\usepackage{amsfonts}
\usepackage[ansinew]{inputenc}
\usepackage[dvips]{epsfig}
\usepackage{graphicx}
\usepackage[english]{babel}
\usepackage{hyperref}

\usepackage{cite}
\usepackage{graphicx}
\usepackage{amscd}
\usepackage{color}
\usepackage{bm}
\usepackage{enumerate}

\usepackage{verbatim}
\usepackage{hyperref}
\usepackage{amstext}
\usepackage{latexsym}
\theoremstyle{plain}
\newtheorem{Theorem}{Theorem}[section]

\newtheorem{Proposition}[Theorem]{Proposition}
\newtheorem{Lemma}[Theorem]{Lemma}
\newtheorem{Remark}[Theorem]{Remark}

\numberwithin{Theorem}{section}
\numberwithin{equation}{section}

\def\square{\vbox{
\hrule height .4pt \hbox{\vrule width .4pt height 7pt \kern 7pt
\vrule width .4pt} \hrule height .4pt }}
\def\QED{\hfill {$\square$}\goodbreak \medskip}

\newcommand{\average}{{\mathchoice {\kern1ex\vcenter{\hrule height.4pt
width 6pt depth0pt} \kern-9.7pt} {\kern1ex\vcenter{\hrule
height.4pt width 4.3pt depth0pt} \kern-7pt} {} {} }}

\def\R{\mathbb{R}}

\renewcommand{\a }{\alpha }
\renewcommand{\b }{\beta }

\newcommand{\D }{\Delta }

\newcommand{\e }{\varepsilon }

\renewcommand{\l }{\lambda }

\newcommand{\n }{\nabla }
\newcommand{\vp }{\varphi }
\newcommand{\rh }{\rho }

\newcommand{\s }{\sigma }

\renewcommand{\t }{\tau }

\renewcommand{\o }{\omega }
\renewcommand{\O }{\Omega }

\newcommand{\ov}{\overline}

\newcommand{\be}{\begin{equation}}
\newcommand{\ee}{\end{equation}}

\newcommand{\de}{\partial}

\newcommand{\ra}{{\rangle}}
\newcommand{\la}{{\langle}}

\renewcommand{\textbf}[1]{\begingroup\bfseries\mathversion{bold}#1\endgroup}

\newcommand{\N}{\mathbb{N}}

\newcommand{\Z}{\mathbb{Z}}

\newcommand{\cB}{{\mathcal B}}

\newcommand{\cH}{{\mathcal H}}

\newcommand{\cL}{{\mathcal L}}

\newcommand{\cO}{{\mathcal O}}
\newcommand{\cP}{{\mathcal P}}

\newcommand{\cU}{{\mathcal U}}
\newcommand{\cV}{{\mathcal V}}

\renewcommand{\epsilon}{\varepsilon}

\begin{document}

\title[Overdetermined  problems with sign-changing eigenfunctions ] 
{Overdetermined  problems with sign-changing eigenfunctions in unbounded periodic  domains }



\author{Ignace Aristide Minlend}
\address{I.A.M: Faculty of Economics and Applied Management, University of Douala,  BP 2701, Douala, Littoral Province, Cameroon}
\email{\small{ignace.a.minlend@aims-senegal.org} }




\begin{abstract}
We prove the existence of  nontrivial unbounded  domains  $\O$ in  the Euclidean space  $\R^d$ for which the Dirichlet eigenvalue problem for the Laplacian on $\Omega$ admits  sign-changing  eigenfunctions with constant Neumann  values on $\partial \Omega$. We  also  establish a similar result by studying  a  partially overdetermined problem on  domains with two boundary components and opposite  Neumann  boundary values. The domains we construct  are periodic in some variables and radial in the other variables, and they  bifurcate from straight (generalized) cylinder or slab. 
\end{abstract}
\maketitle


\maketitle

\section{Introduction and main result}

This paper  is concerned with  the  existence of periodic  sign-changing solutions to  some prototypes of   overdetermined  elliptic boundary value problems in  nontrivial  unbounded domains of the Euclidean space $\R^d$, $d\geq 2$. In the recent year, many  works have been devoted  to the study  of the overdetermined  problem \begin{align}\label{probcafa}
-\Delta u = f(u) \quad \text{in $\Omega$}, \qquad u=0, \quad \partial_\eta u=\textrm{const}  \qquad \text{on $\partial \Omega$},
\end{align} where  $f: [0,\infty) \to \R$ is  a locally Lipschitz function and  $\eta$ is the unit outer to the boundary.

In  1971, Serrin  studied  the case  $f\equiv1$  in the pioneer paper \cite{Serrin}  and proved by Alexandrov \cite{Alexandrov} moving plane method that the only bounded and
regular domains  in the Euclidean space  $\mathbb{R}^d$, $d\geq 2$, where \eqref{probcafa} is solvable are balls.  Soon after  this  celebrate result  was  communicated to the PDE community, several authors have developed interest in  the study of symmetry properties   as well as rigidity  results  related  to problem 
\eqref{probcafa}.  We refer the reader to   \cite{Alessandrini, Gazzola, Greco,  Philippin,  GarofaloLewis, Prajapat, PaynePhilippin, Lamboley, FragalaGazzola, BrockHenrot, Reichel, W. Reichel, BerchioGazzolaWeth, FraGazzolaKawohl,farina-valdinoci,farina-valdinoci:2010-1,farina-valdinoci:2010-2,farina-valdinoci:2013-1, FarimaKawohl, farina-valdinoci:2013-2, BCNI}.  In 1997,  Berestycki, Caffarelli and Nirenberg \cite{BCNI} conjectured that, if  $\O$  is a  domain such that $\R^{d}\setminus \ov \O$ is connected, the existence  of a bounded positive solution to problem  \eqref{probcafa} implies  that  $\O$  must  be a half-space, a ball, the complement of a ball, or a circular-cylinder-type domain $\R^j\times C$ (up to rotation and translation), where $C$ is a ball or a complement of a ball in $\R^{d-j}.$  For  $f(u)=\lambda_1 u$,  where  $\lambda_1$
is the first eigenvalue of the Laplacian with $0$-Dirichlet boundary
condition, this conjecture was disproved in dimension  $d\geq 3$ by
Sicbaldi \cite{Sic},  and later in dimension $d\geq2$ by  Sicbaldi
and  Schlenk in \cite{ScSi}, where they proved existence of periodic
and unbounded extremal domains  bifurcating from straight cylinder
$B_1\times \R$.  Subsequently,  Fall, Weth and the author  studied  case  $f\equiv 1$ in  \cite{ Fall-MinlendI-Weth} by constructing  periodic unbounded domains bifurcating from generalized-type cylinder domains in $\R^d$. Further results addressing    \eqref{probcafa}  in spaces forms can be found in   \cite{Fall-MinlendI-Weth2, Ku-Pra, morabito-sicbaldi}.

It is important  to note  that  the results  in  the previous  works  all assume a sign on the solution, while only few results appear in the literature  regarding the  existence of sign-changing solutions in the context of overdetermined boundary value problems. In fact we  are only able  to  cite  the   contributions \cite{Ramm, Deng, BCanutoDRial, B.Canuto,Fall-MinlendI-Weth4, Ruiz-arxiv} addressing  \eqref{probcafa} in bounded domains  and for particular functions $f$. In  particular, \cite{BCanutoDRial} considers   
\begin{align*}
  \Delta u+\o^2 u=-1 \quad \text{in $\Omega$}, \quad u=0, \quad \partial_\eta u=\textrm{const}  \quad \text{on $\partial \Omega$}
\end{align*}
 and proves   under suitable assumptions on $\o\in \R$ that the  only bounded domain $\O$ such that there exists a solution 
is the ball $B_1$, independent  on the sign of $u$, provided  $\partial \O$ is a perturbation of the unit sphere $\partial B_1$ in $\R^d$. A similar result   was derived  in \cite{B.Canuto} by considering a different Neumann boundary condition.  Moreover, in the  work \cite{Ruiz-arxiv}, Ruiz is considered  the problem \eqref{probcafa} with a specific nonlinearity $f$  and proved  the existence of sign changing solutions  to the  problem \eqref{probcafa} in perturbations of the unit ball $B_1$. In contrast, less is known regarding the existence of sign-changing solutions to problem  \eqref{probcafa}  in nontrivial  \emph{ unbounded}   domains. To our level of information, we can only  quote the recent   contribution \cite{Fall-MinlendI-Weth4} by Fall, Weth and the author where we  proved the  existence of a family of unbounded   subdomains  $\O$ bifurcating from that flat cylinder $B_1\times \R$  and for which the Neumann eigenvalue problem for the Laplacian on $\Omega$ admits sign-changing eigenfunctions with constant Dirichlet values on $\partial \Omega$.

In this paper, we deal the existence of sign-changing  solutions  in  nontrivial  \emph{ unbounded}   domains by considering  two   prototypes of problem  \eqref{probcafa}.  The  first problem we study  is the Dirichlet eigenvalue overdetermined problem
\begin{equation} \label{eq:Dirichlesign}
  \left \{
    \begin{aligned}
       -\Delta u &=  \lambda u && \qquad \text{in $\Omega$,}\\
              u &=0 &&\qquad \text{on $\partial \Omega$,}\\
             \frac{\partial u}{\partial \eta}  &= c &&\qquad \text{on $\partial \Omega,$}
    \end{aligned}
 \right.
\end{equation}
where $c>0$, $\lambda>0$ and $\eta$ is the outer normal  vector field a the boundary. In  Section \ref{Parrtialy},  we treat a partially overdetermined problem of the form
\begin{equation} \label{eq:perturbed-partially}
  \left \{
    \begin{aligned}
       -\Delta u &=  \mu u && \qquad \text{in $\Omega$,}\\
              u &=0 &&\qquad \text{on $\partial \Omega$,}\\
             \frac{\partial u}{\partial \eta}  &= \pm \gamma  &&\qquad \text{on $\partial \Omega^{\pm}$,}\\
    \end{aligned}
       \right.
     \end{equation}
for  some $\mu >0$ and $\gamma >0$. 
Here, $$ \partial \Omega^+= \{ (x, t) \in \partial \Omega, \quad t> 0  \}\quad \textrm{and} \quad\partial \Omega^-= \{ (x, t) \in \partial \Omega,  \quad t< 0  \}.$$  The strategy we use  for this work  allows  us to deduce the existence of sign-changing solutions for related problems   to  \eqref{eq:Dirichlesign} and  \eqref{eq:perturbed-partially},   with  specific  \emph{non constant} Neumann boundary conditions.  

To state our first main result, we fix $\alpha \in (0,1)$ and define  by  $C^{2,\alpha}_{p, e}(\R^m)$  the space of  even  and  $2\pi \Z^m$ -periodic $C^{2,\alpha}$-functions on $\R^m$, and we let $\cP^{2,\alpha}_{p, e}(\R^m)$ denote the open subset of  $C^{2,\alpha}_{p, e}(\R^m)$  made of strictly positive functions which are  invariant with respect to  coordinate permutations.  For a function $h \in \cP^{2,\alpha}_{p, e}(\R^{m})$, we  define  the domain               
\begin{equation}\label{eq:PertTorus}
\Omega_h:= \left\{\left(t,x \right)\in  \R^N\times \R^{m} \::\: |t|<\frac{1}{h(x)}  \right\}\subset \R^{N+m}.
\end{equation}

\begin{Theorem}\label{Theo1-ND}
For each  $N, m, n  \in \N$ be positive integers.  Then there exist  $\e_n>0$ and (explicit) constants $\mu_n, \kappa_n, c_n >0$, $\beta_n,\delta_n \in \R \setminus \{0\}$, depending only on $N$ and $n$, and a smooth curve
$$
(-{\e_n},{\e_n}) \to   (0,+\infty) \times  \cP^{2,\alpha}_{p, e}(\R^m) ,\qquad s \mapsto (\mu^n_s,h^n_s)
$$
with $\mu^n_s \big|_{s=0}= \mu_n$, 
$$
h^n_s(x)= \kappa_n\sqrt{\mu^n_s} + s \beta_n \vartheta(x) + o(s) \qquad \text{as $s \to 0$ uniformly on $\R$,}
$$
where  $$\vartheta(x):=\cos(x_1)+\cdots+\cos(x_m)$$
and the property that the overdetermined boundary value problem
\begin{equation}\label{eq:solved-main-ND}
  \left \{
    \begin{aligned}
       \D w^n_s+ {\mu^n_s} w_s &=  0 && \qquad \text{in $ \Omega_{ h^n_s}$,}\\
             w^n_s&=0 &&\qquad \text{on $\partial \Omega_{ h^n_s}$,}\\
              |\n w^n_s | &=c_n {\sqrt{\mu^n_s}} &&\qquad \text{on $\partial  \Omega_{ h^n_s}$}
    \end{aligned}
       \right.
\end{equation}
admits a classical solution $w^n_s$ for every $s \in (-\e_n,\e_n)$ which is radial in $t$, even in $x_1,\dots, x_m$, ${2\pi} $-periodic in
$x_1,\dots, x_m$ and invariant with respect to permutations of the
variables $x_1,\dots,x_m$. Moreover, we have
\begin{equation}
  \label{eq:w-s-expansion1}
w_s(\frac{t}{h^n_s(x)},x)= U_n(|t|) +s\bigl\{\psi_1(|t|)+  \delta_n \,|t| U_n'(|t|)\bigr\} \vartheta(x) + o(s)  \quad \text{as $s \to 0$}
\end{equation}
uniformly on $B_1 \times \R^m$, where $t \mapsto U_n(|t|)$ is the $n$-th nonconstant radial Dirichlet  eigenfunction of the Laplacian on the unit ball $B_1$ of $\R^N$, and $t \mapsto \psi_1(|t|)$ is a  suitable radial function defined on  the unit ball $B_1$.
\end{Theorem}
Before we state our second main result,  we make the following  observations regarding Theorem \ref{Theo1-ND}. 

\begin{Remark}\label{eq:RKmain-ND}
Let $J_{\nu}$ denote the Bessel function of the first kind of order $\nu>-1$, and let
$$
0< j_{\nu,1} < j_{\nu,2} < j_{\nu,3} < \dots
$$
denote the ordered sequence of zeros of $J_\nu$. We put   $I_{\nu }(r):= r^{-\nu }J_{\nu }( r)$ and  let  $\sqrt{\nu_1}$ is the first positive zero of 
$$
z\mapsto \frac{z J_{N/2}(z)}{J_{N/2-1}(z)}-(N-1). 
$$
Then constants  in Theorem  \ref{Theo1-ND}  are given by 
\begin{align*}
\mu_n& = \frac{j^2_{N/2-1,n}}{j^2_{N/2-1,n} - \nu_1},\qquad   \kappa_n = \frac{1}{j_{N/2-1,n}}, \qquad   c_n=I'_{N/2-1}(j_{N/2-1,n})  \\
\beta_n &=\frac{I_{N/2-1} ( \sqrt{\nu_1}) }{j_{N/2-1,n} I_{N/2-1}'(j_{N/2-1,n})\sqrt{j_{N/2-1,n}^2-\nu_1}},\qquad  \delta_n=-\frac{I_{N/2-1} ( \sqrt{\nu_1}) }{ j_{N/2-1,n} I_{N/2-1}'(j_{N/2-1,n})}. 
\end{align*}
Moreover, the functions $U_n$  and  $\psi_1$ in Theorem~\ref{Theo1-ND} is defined  by 
\begin{equation*}
r \mapsto U_n(r)= I_{N/2-1}(j_{N-1/2,n}\, r)\quad \textrm{and}\quad  r \mapsto \psi_1(r)= I_{N/2-1} (r \sqrt{\nu_1}).
\end{equation*}
In comparing   Theorem  \ref{Theo1-ND}  with the work  by   Schlenck and Sicbaldi \cite{ScSi} for  positive eigenvalue problem and for $m=1$,  we underline  the  extremal domains  for  the first eigenvalue of the Dirichlet Laplacian  in  \cite{ScSi}  bifurcate from the straight cylinder $\partial B_1\times \R$  with a  period 
$$T_{*}(N)=\frac{2\pi}{  \sqrt{j_{N/2-1,1}^2-\rho^2_{N/2-1,1}} },$$  where  $\rho_{N/2-1,1}$ is the unique zero of $$z\mapsto z  J_{N/2-2}(z)+J_{N/2-1}(z)$$ in the interval $(0, j_{N/2-1,1})$. 

In our case,   bifurcations are $2\pi$-periodic and  occur at the cylinder radius 
$$R_*(N)= \sqrt{j_{N/2-1,1}^2-\nu_{1}}.$$  
\end{Remark}

We now turn our attention  on problem \eqref{eq:perturbed-partially}. We stress that in  contrast to problem \eqref{eq:solved-main-ND}, where the solution  $w$ is assume  to be radial in the $t$ variable, we  require the  solution of  \eqref{eq:perturbed-partially} to  be odd  in this variable  for the condition   $\frac{\partial u}{\partial \eta}  = \pm 1\quad \text{on $\partial \Omega^{\pm}$}$ to hold. This leads to our second main result.
\begin{Theorem}\label{Theo2-DN}
For each  $N, m, n  \in \N$ be positive integers.  Then there exist  $\rho_n>0$ and (explicit) constants $a_n, b_n, d_n >0$, depending only on $N$ and $n$, and a smooth curve
$$
(-{\rho_n},{\rho_n}) \to   (0,+\infty) \times  \cP^{2,\alpha}_{p, e}(\R^m) ,\qquad s \mapsto (\widetilde{\mu}^n_s,\widetilde{h}^n_s)
$$
with $ \widetilde{\mu}^n_s \big|_{s=0}=d_n$, 
$$
\widetilde{h}^n_s(x)= a_n\sqrt{\widetilde{\mu}^n_s} + s b_n \vartheta(x)  + o(s) \qquad \text{as $s \to 0$ uniformly on $\R$}
$$
with 
$$\vartheta(x):=\cos(x_1)+\cdots+\cos(x_m)$$  and the property that the overdetermined boundary value problem
\begin{equation}\label{eq:solved-main-ND2}
  \left \{
    \begin{aligned}
       \D \widetilde{w}^n_s+ \widetilde{\mu}^n_s \widetilde{w}^n_s &=  0 && \qquad \text{in $ \Omega_{\widetilde{h}^n_s}$,}\\
             \widetilde{w}^n_s &=0 &&\qquad \text{on $\partial \Omega_{\widetilde{h}^n_s}$,}\\
 \frac{\partial  \widetilde{w}^n_s}{\partial \eta_s}  &= \pm\frac{1}{\sqrt{\gamma_n(s)}}  &&\qquad \text{on $\partial  \Omega^{\pm}_{\widetilde{h}^n_s}$}
    \end{aligned}
       \right.
\end{equation}
admits a classical solution $\widetilde{w}^n_s$ for every $s \in (-\rh_n,\rho_n)$ which is  odd  in $t$, even in $x_1,\dots, x_m$, ${2\pi} $-periodic in $x_1,\dots, x_m$ and invariant with respect to permutations of the
variables $x_1,\dots,x_m$.  Here $ \eta_s$ denotes the  unit outer normal vector filed to the  boundary $\partial  \Omega_{\widetilde{h}^n_s}$. Moreover, we have
\begin{equation}
  \label{eq:w-s-expansion}
\widetilde{w}^n_s (\frac{t}{\widetilde{h}^n_s(x)},x)= v_n(t) + s \Bigl(\sin(\frac{\pi t}{2} )- (-1)^n \cos( n\pi t \Bigr)\vartheta(x)  + o(s)\quad \text{as $s \to 0$}
\end{equation}
uniformly on $(-1, 1) \times \R^m$, where $t \mapsto v_n(t)=(-1)^n\sin(n \pi t)$ is the $n$-th nonconstant  Dirichlet  eigenfunction of the Laplacian on  $(-1, 1) \subset \R.$ 
\end{Theorem}

\begin{Remark}\label{eq:RKmain-ND2}
The constants  in Theorem  \ref{Theo2-DN}  are given by
\begin{align*}
d_n& = \frac{n^2 }{n^2- \frac{1}{4}},\qquad   a_n = \frac{1}{n\pi},\qquad  b_n =\frac{1}{\sqrt{n^2\pi^2- \frac{\pi^2}{4}}} \qquad  \textrm{and}\quad \gamma_n(0)=n^2\pi^2- \frac{\pi^2}{4}.
\end{align*}
We note that sets of similar shape that the one of Theorem  \ref{Theo2-DN} were also obtained  in \cite{Mi.Al.Th.}, where  Thiam, Niang and the author constructed bifurcating  hypersurfaces  with constant nonlocal mean curvature.
\end{Remark}

The proof of Theorem \eqref{Theo1-ND}  is achieved by the use of Crandall-Rabinowitz bifurcation theorem, \cite{M.CR}. Our aim is  solve  the problem  \eqref{eq:Dirichlesign}  on the domain  $\O_h$  given by  \eqref{eq:PertTorus}. In Section  \ref{sectio1} we transform \eqref{eq:Dirichlesign}  to   the equivalent problem \eqref{eq:perturbed-strip-ND-Direqui1} on the fixed domain $\O_*=B_1\times \R^m$. Under the fonctional setting of Section \eqref{eq:funcsetline},  \eqref{eq:perturbed-strip-ND-Direqui1} can be reformulated  to an operation equation  $F_\lambda(u, h)=0$ between suitable Banach spaces with unknown functions $u \in C^{2,\alpha}_{p,rad}(\overline{\O_*})$ and $h \in C^{2,\alpha}_{p,e}(\R^m)$ for some $\alpha \in (0,1)$.
Here  $C^{2,\alpha}_{p,rad}(\overline{\O_*})$ denotes the space of $C^{2,\alpha}$-functions $u=u(t,x)$ which are radial in $t$ and $2\pi$ periodic and even in each of the variables  $x_1, \dots, x_m$.  By the help of the Remark \eqref{eqrresolutionform}, we  are led to  reducing  the equation  $F_\lambda(u, h)=0$ to an equivalent  the type $G_\lambda(u)=0$ for some function $(\lambda,u) \mapsto G_\lambda(u)$, see \eqref{eq bifurcate}.  In Section~\ref{analyis}, we analyse  the linear  operator $D_uG_\lambda(0):  X_2 \rightarrow  X_0\times Z_1$ computed in Proposition  \ref{eqlinearised}.   In order to get a one dimensional kernel, we needed  to restrict the mapping $G_\lambda$ on the space of functions $u(t, x)$  which are invariant under permutations of  coordinates in $\R^m$.  Next applying  Fredholm's alternative \cite[Theorem 2.3]{Kress}, we  show that  $D_uG_\lambda(0):  X_2 \rightarrow  X_0\times Z_1$  has codimension one and satisfies  the  transversality condition  in the Crandall-Rabinowitz bifurcation theorem \cite{M.CR}. 

The proof of Theorem   \ref{Theo2-DN} follows  similar steps with the slight difference that instead of  radial functions $u(t, x)$ in the variable $t$, we need to  work on  the space  of odd  functions in the variable $t$. 

We  close  this introduction by highlighting  as explained in   Remark  \ref{noncostNeum}, the existence of sign-changing solutions to the Dirichlet  problem  in \eqref{eq:Dirichlesign}  with a  specific  non constant Neumann boundary  value  involving   the boundary parameter $h$  in  \eqref{eq:PertTorus} and expressed by $c(x):=g(h(x))$,  for some function $g:(0, +\infty)\rightarrow (0, +\infty) $ see  \eqref{eq:Neumandaat2}. One could then ask  for the  class of functions $g$ such that  the Dirichlet problem in \eqref{probcafa} admits a solution with   a non constant Neumann boundary  value  involving   the boundary parameter of a perturbed  domain. This question is left open. 

The  paper ends  with Section \ref{eq: Cradal Rabi1},  where  we  state  the  Crandall-Rabinowitz bifurcation theorem for the reader convenience.

\bigskip
\noindent \textbf{Acknowledgements}: 
This work was  carried  out when  the author  was visiting the Institute of Mathematics and Informatics of the Goethe University Frankfurt as a Humboldt postdoctoral fellow. He is gratefully to the Humboldt Foundation for funding his research  and wishes to thank Department of Mathematics  of  the Goethe-University Frankfurt  for the hospitality. The author  also thanks his host Prof. Tobias Weth and Prof. Mouhamed Moustapha Fall for their  helpful suggestions  and  comments  throughout
the writing of this paper.

\section{The pull back of problem \eqref{eq:Dirichlesign}}\label{sectio1}
For a function $h \in \cP^{2,\alpha}_{p,e}(\R^{m})$, we  define  the domain               
\begin{equation}\label{eq:PertTorus3}
\Omega_h:= \left\{\left(t,x \right)\in  \R^N\times \R^{m} \::\: |t|<\frac{1}{h(x)}  \right\}\subset \R^{N+m}.
\end{equation}

In our first result (Theorem  \ref{Theo1-ND}), we look for  a constant $\mu >0$ and  and  nontrivial sign changing solutions $u$  to the problem
\begin{align*}
(\textrm{D}_{\mu}):   
  \left \{
    \begin{aligned}
        \D u +\mu  u &=  0 && \qquad \text{in $\Omega_h$,}\\
              u &=0 &&\qquad \text{on $\partial  \Omega_h$}\\
               \frac{\partial u}{\partial \eta_h } &=c &&\qquad \text{on $\partial \Omega_h$,}
    \end{aligned}
       \right.
\end{align*}
where  $\eta_h$ is the unit outer normal to the boundary $\partial \Omega_h$ and $c>0$. 

To  solve $(\textrm{D}_{\mu})$, consider the Dirichlet problem   
\begin{equation}
  \label{eq:perturbed-strip-ND-Dir}
  \left \{
    \begin{aligned}
  L_{\l,n} v &=  0 && \qquad \text{in $\Omega_{h}$,}\\
           v &=0 &&\qquad \text{on $\partial \Omega_{h}$,}
    \end{aligned}
       \right.
     \end{equation}
where
\begin{equation}
  \label{eq:opLlam-Dir}
  L_{\l,n}:= 
\D_{\t} + \lambda \Delta_{x}  +  j_{N/2-1,n}^2 \textrm{id}   
 \end{equation}
and   $(j_{\b,n})_n$ are the increasing positive zeros of the Bessel function $J_{\b}$.  We emphasise that   if $v$ is a solution of  \eqref{eq:perturbed-strip-ND-Dir} in $\Omega_h$, then the function 
\begin{equation}
  \label{eq:changsoluD}
w^{\lambda} (t, x):= v( t/\sqrt{\lambda}, x)
\end{equation}
solves  the  Dirichlet problem  in   $(\textrm{D}_{\mu})$  with $\mu=\frac{j_{N/2-1,n}^2 }{\lambda}$  on the domain  $\O_{\frac{h}{ \sqrt{\lambda}}}$.  Furthermore, defining 
 \begin{equation}\label{eq:Bessel}
I_{\nu }(r):= r^{-\nu }J_{\nu }( r),
 \end{equation}
we have a  solution  $u_n(t, x):=I_{N/2-1}(j_{N/2-1,n}|\t|)$ to  \eqref{eq:perturbed-strip-ND-Dir} (in the case  $h\equiv 1 $),  which satisfies  $ \D_{\t} u + j^2_{N/2-1,n} u=0$
and $$u_n(t, x)=0, \qquad \n u_n(\t, x) \cdot \t = j_{N/2-1,n} I'_{N/2-1}(j_{N/2-1,n})     \quad \textrm{on}\quad \partial B_1\times \R^m.$$
In particular   \eqref{eq:perturbed-strip-ND-Dir} has, for every fixed $\l>0$, a sequence of solutions  given by $(u_n)_n$ on  $\O_*:=B_1\times\R^m$. \\

We now express the normal derivative of  $w^{\lambda}$ in term of  the function $v$.  We note the outer unit  normal on $\partial \Omega_h$ with respect to the Euclidean
metric $g_{eucl}$  is given by  $\eta_h: \partial \O_h \to  \R^{N+m}$, with 
\begin{equation}
  \label{eq:def-mu-phi}
\eta_h (t, x) = \dfrac{ 1}{\sqrt{1+ \frac{ |\nabla h(x)|^2 }{ h^4(x)} }} \Bigl(\dfrac{t}{|t|},  \dfrac{ \nabla h(x) }{ h^2(x)}, \Bigl)   \in  \R^{N+m} \qquad \text{for
$(t, x) \in \partial \Omega_h.$}
\end{equation}
Then 
\begin{align} \label{eq:derinormalw}
  \frac{\partial w^\lambda }{\partial\eta_{ \frac{h}{ \sqrt{\lambda}} }}(t, x) = \dfrac{ 1}{\sqrt{1+ \lambda \frac{|\nabla h(x)|^2}{ h^4(x)}}} \Bigl[\n_t w^\lambda (t, x)\cdot  \dfrac{t}{|t|} +  \sqrt{\lambda }\frac{ \n_xh(x) }{ h^2(x)} \cdot \n_x w^\lambda (t, x) \Bigl].
\end{align}
Since we require $w^{\lambda} (t, x)=0$ on  $\partial \Omega_{h/\sqrt{\lambda}}=\left\{\left(t,x \right)\in  \R^N\times \R^m  \::\: |t|=\frac{ \sqrt{\lambda }}{h(x)}  \right\},
$
assuming   $w^{\lambda}(t, x)$ is a radial function in the $t$ variable, we  have  
 $w^{\lambda} ( \frac{ \sqrt{\lambda }}{h(x)} e_1, x)=0$ for all $x\in \R^m$ and differentiating  this with respect to $x$,  we find 
$$
  \n_x w^\lambda  ( \frac{ \sqrt{\lambda }}{h(x)} e_1, x)= \sqrt{\lambda } \n_t w^\lambda ( \frac{ \sqrt{\lambda }}{h(x)} e_1, x)\cdot  e_1 \frac{ \n h(x) }{ h^2(x)} 
$$ 
This with  \eqref{eq:derinormalw} provides 
\begin{align}\label{eq:derfin}
  \frac{\partial w^\lambda }{\partial\eta_{ \frac{h}{ \sqrt{\lambda}} }}( \frac{ \sqrt{\lambda }}{h(x)} e_1, x) &=\sqrt{1+ \lambda\frac{ |\nabla h(x)|^2 }{ h^4(x)} }\n_t w^\lambda  ( \frac{ \sqrt{\lambda }}{h(x)} e_1, x)\cdot e_1 =  \frac{1}{\sqrt{\lambda}} \sqrt{1+ \lambda\frac{ |\nabla h(x)|^2 }{ h^4(x)} } \n_t v ( \frac{1}{h(x)} e_1, x)\cdot e_1. 
\end{align}

From  \eqref{eq:derfin}, a radial function  the  $v$  in $t$ solves 
\begin{equation}\label{eq:per}
  \left \{
    \begin{aligned}
& L_{\l,n} v =  0 && \qquad \text{in $\Omega_h$,}\\
          & v =0 &&\qquad \text{on $\partial \Omega_h$,}\\
     &\sqrt{1+ \lambda\frac{ |\nabla h(x)|^2 }{ h^4(x)} } \n_t v ( \frac{1}{h(x)} e_1, x)\cdot e_1 = c_1 && \qquad  \text{  $ x\in \R^m$,} \\
    \end{aligned}
       \right.
     \end{equation}
if and only if the function  $w^\lambda$ in \eqref{eq:changsoluD}  solves   the problem $ (\textrm{D}_{\mu})$   with  
\begin{equation}\label{Neumanndata}
c= \frac{c_1}{\sqrt{\lambda}}.
  \end{equation}
   Note that  when $h=1$, \eqref{eq:per}  is solved by the function $u_n(t, x):=I_{N/2-1}(j_{N/2-1,n}|\t|)$, with 
 \begin{equation}\label{constc1}  
c_1=\n_t u_n(e_1, x) \cdot  e_1 =  j_{N/2-1,n} I'_{N/2-1}(j_{N/2-1,n}).
 \end{equation}
We  pull back problem \eqref{eq:per} on  the fixed unperturbed domain $\Omega_*$  using the parametrization  $$ \Psi_h: \Omega_*  \to  \Omega_h , \quad  (\t, x)  \mapsto (t, x)=(\frac{\t}{h(x)}), x),$$  with  inverse given by 
$ \Psi^{-1}_h:  \Omega_h   \to  \Omega_*, \quad  (t, x)  \mapsto   (h(x) t, x)$.\\ 
We  then  consider  the ansatz 
\begin{align}\label{eqans1}
v(t, x)= u(h(x)t, x)=u(\tau, x) \qquad \text{for some function $u: \Omega_* \to \R$.}
\end{align}
 and determine the differential operator $L^h_\lambda$ with the property 
\begin{align}\label{eqreladiffopets}
[L_{\l,n}^h u] (h(x)t, x) = [L_\lambda v](t,x) \qquad \text{for $(x,t) \in \Omega_h$}
\end{align}  
By a straightforward computation, we obtain,
\begin{align} \label{eq:reldiffope-ND-Dir}
L_{\l,n}^h u(\t,x) =&j^2_{N/2-1}u(\t,x) + \lambda \Delta_{x}u(\t,x) +   h^2(x)\D_\t u(\t,x)+ \lambda     \frac{|\n h(x)|^2}{h(x)^2}   \n^2u(\t,x)[\t,\t]   \nonumber\\
&+   \frac{ 2 \lambda  }{h(x)} \n_xh(x) \cdot   \n_x (\n_{ \tau }u(\t,x)\cdot \t) +  \lambda \frac{ \Delta h(x) }{h(x)}  \n u(\t,x)\cdot \t. 
\end{align}

With  this, problem \eqref{eq:per}  is therefore equivalent to 
\begin{equation}
  \label{eq:perturbed-strip-ND-Direqui1}
  \left \{
    \begin{aligned}
&L_{\l,n}^{h} u =  0 && \qquad \text{in $\Omega_*$,}\\
          & u =0 &&\qquad \text{on $\partial \Omega_*$,}\\
     & h \sqrt{1+\lambda \frac{ |\nabla h(x)|^2 }{ h^4(x)} } \n_\tau u(e_1, \cdot) \cdot e_1 =\n u_n (e_1,\cdot) \cdot  e_1 &&\qquad \text{in  $\R^m$.}
    \end{aligned}
       \right.
\end{equation}

\section{Functional setting }\label{eq:funcsetline}
In the following, we set 
$$
C^{k,\a}_{p,rad}(\overline \Omega_*):= \{ u \in C^{k,\alpha}(\overline \Omega_*)\::\: \text{$u$ is radial in ${\t}$, $2\pi \Z^m$- periodic and  even in $x$ } \},
$$
endowed with the norm $
u \mapsto \|u\|_{C^{k,\alpha}}:= \|u\|_{C^{k,\alpha}(\ov{\O_*})} .
$
$$
X_k:= C^{2,\a}_{p,rad}(\overline \Omega_*)
$$
as well as 
\begin{align*}
C^{k,\a}_{p,e}(\R^m):= \{z \in C^{k,\alpha}(\R^m)\::\: \text{$z$ is $2\pi \Z^m$- periodic  and  even in $x$  }\},
\end{align*}
\begin{align*}
Z^1:= C^{1,\a}_{p,e}(\R^m).
\end{align*}
We also set
\begin{align*}
Y_2^+:= \{h \in C^{2,\alpha}_{p,e}(\R^m)\::\: \text{ $h>-1$ }\} 
\end{align*}
$$
X_2^D:=\{u \in X_2 \::\: \text{$u= 0$ on $\partial \Omega_*$}\},
$$
 and define  $$ K_{\lambda}(u, h):= h \sqrt{1+\lambda \frac{ |\nabla h(x)|^2 }{ h^4(x)} } \n_\tau u(e_1, \cdot) \cdot e_1 -\n u_n (e_1,\cdot) \cdot  e_1$$
and 
 $$
F_\lambda:  X_2^{D} \times Y_2^+ \to X_0  \times Z_1  , \qquad (H_\lambda(u, h), Q_{\lambda}(u, h))
$$
where 
$$ H_\lambda(u, h):= L_{\lambda, n}^{1+h} (u+u_n)\quad \textrm{and}\quad Q_{\lambda}(u, h):= K_{\lambda}(u+u_n, 1+h).$$
By construction if 
\begin{equation}\label{eq: eqfirst}
F_\lambda(u, h)=0,
\end{equation}
then the  $\tilde{u}:=u+u_n$ solves \eqref{eq:perturbed-strip-ND-Direqui1} with $h$ replaced by $1+h$. 
We further reduce the equation  \eqref{eq: eqfirst} to a single unknown  $u$ by  eliminating  the variable $h$ in the following remark.
\begin{Remark}\label{eqrresolutionform}
Since $ (t,x) \mapsto  u_n(t, x): =I_{N/2-1}(j_{N/2-1,n}|\t|)$ solves  $ L_{\l,n} u_n=0$ in  $ \R^{N+m} \supset \O_{1+h}$,  we have by  \eqref{eqreladiffopets} that  the function
\begin{equation}\label{eq:approxisolu}
u_n(\t/(1+h))=u_n-h \n u_n \cdot \t+ O(||h||_{C^{2,\alpha}(\R^m)}^2)
\end{equation}
solves 
\begin{equation}\label{eq:approxim}
L_{\l,n}^{1+h} u_n(\t,x) =  0  \qquad \text{in $\Omega_*$}.
\end{equation}
It is then reasonable to construct  solutions  to   \eqref{eq:perturbed-strip-ND-Direqui}  on   $\O_{1+h}$  as an approximation of  the linear part in  \eqref{eq:approxisolu}. That is a solution on the form 
\begin{align}\label{eq:Approx}
U(\t, x):=u+u_n-h \n u_n \cdot \t,
\end{align}
 with $u$ and $h$ small. Note that, for this function to satisfy zero Dirichlet and constant Neumann boundary equating $\n u_n\cdot t$ on $\de \O_*$, we must have 
\begin{equation}\label{eq:foermh1}
h=h_u=\frac{ u (e_1,\cdot)  }{\n u_n(e_1) \cdot e_1}. 
\end{equation}
We have 
\begin{align}
 \n U (e_1,\cdot) \cdot  e_1&=  \n u(e_1,\cdot) \cdot  e_1+\n u_n (e_1,\cdot) \cdot  e_1- h\left( \n u_n (e_1 ) \cdot  e_1+  \n^2 u_n (e_1)[e_1,e_1] \right) \nonumber\\
 &= \n u(e_1,\cdot) \cdot  e_1+\n u_n (e_1,\cdot) \cdot  e_1+(N-2)h  \n u_n (e_1 ) \cdot  e_1\nonumber\\
 &= \n u(e_1,\cdot) \cdot  e_1+\n u_n (e_1,\cdot) \cdot  e_1+(N-2)  u (e_1, \cdot),
\end{align}
where we have used \eqref{eq:foermh1} and  the relation  $\nabla u \cdot t + {\rm Hess}(u)(t,t) = |t|^2 \Delta_t u  -(N-2)\nabla u \cdot t$ for any  radial function $u$. 
The third  condition  in \eqref{eq:perturbed-strip-ND-Direqui} with $h$ replaced by $1+h$   reads 
\begin{align} \label{eq:spacecondi}
& \sqrt{1+ \lambda\frac{ |\nabla h_u(x)|^2 }{ (1+h_u(x))^4} } ( 1+ h_u(x)) \Bigl(\n u(e_1,\cdot) \cdot  e_1+\n u_n (e_1,\cdot) \cdot  e_1+(N-2)  u (e_1, \cdot)\Bigl) \nonumber\\
&=\n u_n (e_1,\cdot) \cdot  e_1.
\end{align}
\end{Remark}
We now  consider  the open set
$$
\cU := \left\{u \in X_2\::\: \, \frac{ u (e_1,\cdot)  }{\n u_n(e_1) \cdot e_1} >-1 \right \},
$$
 and the mapping  
\begin{align}\label{eq:perturbed-strip-ND-Direqui}
G_\lambda: \cU   \to X_0  \times Z_1, \qquad G_\lambda(u):= F_\lambda \circ M(u)=(H_\lambda \circ M(u),  Q_{\lambda}\circ M(u)),
\end{align}  
where 
$M: \cU \subset X_2 \to X_2^{D} \times  Y_2^+ $ is defined by $Mu= (M_1 u,M_2 u)$ with 
\begin{align}\label{eq:defM}
[M_1 u]({\t},x) = u-h_u \n u_n \cdot \t, \quad [M_2 u](\t,x) = h_u. 
\end{align}
It  then follows from the Remark \eqref{eqrresolutionform}  that
\begin{align}\label{eq bifurcate}
G_\lambda(u)=0,
\end{align}
then the function 
\begin{align}\label{eq:solfinal}
M_1 u +u_n
\end{align}
solves the problem  \eqref{eq: eqfirst}  with 
\begin{align}\label{eq exprehh}
h=h_u=\frac{ u (e_1,\cdot)  }{\n u_n(e_1) \cdot e_1}. 
\end{align}
We have
$$
G_\lambda(0) = 0 \qquad \text{for all $\lambda >0$,}
$$
and  by \eqref{eq:Approx}, \eqref{eq:defM}, the definition of $Q_\lambda$ and   $K_\lambda$, we have from   \eqref{eq:spacecondi}

\begin{align}\label{eq: defQ}
Q_{\lambda}\circ M (u)&=\sqrt{1+ \lambda\frac{ |\nabla h_u|^2 }{ (1+h_u)^4} } ( 1+ h_u) \Bigl(\n u(e_1,\cdot) \cdot  e_1+\n u_n (e_1,\cdot) \cdot  e_1+(N-2)  u (e_1, \cdot)\Bigl)\nonumber\\
& -\n u_n (e_1,\cdot) \cdot  e_1.
\end{align}
\begin{Proposition}\label{eqlinearised}
The map $G_\l : \mathcal{U} \cap  X_2 \rightarrow  X_0\times Z_1$  defined by \eqref{eq:perturbed-strip-ND-Direqui} is of class $C^\infty$. Moreover for all $v\in   X_2 $,
\begin{align}\label{eequivla-ND}
D G_\lambda(0)v= \Bigl(j_{N/2-1,n}^2 v+ \lambda \Delta_{x} + \D_\t v,  \n v(e_1, \cdot )\cdot e_1+(N-1)v(e_1, \cdot)\Bigl)
\end{align}
\end{Proposition}
\begin{proof}
The proof of the first statement is achieved once we show that each the  mappings 
$u \mapsto H_\lambda \circ M (u)$ and $u\mapsto Q_{\lambda}\circ M (u) $ are  $C^\infty$. It is clear from the definition in \eqref{eq: defQ} that  $Q_\lambda$ is   $C^\infty$. Furthermore,  the  map  $H_\lambda$ is  $C^\infty$ from its definition using \eqref{eq:reldiffope-ND-Dir}. Since $M$ is linear,  the $C^\infty$-character of the map $u \mapsto H_\lambda \circ M (u)$ follows. 

We now prove \eqref{eequivla-ND}.  By a direction computation, using \eqref{eq: defQ} we find \begin{align}\label{eq diffQ}
D Q_{\lambda}(0)v=\n v(e_1, \cdot )\cdot e_1+(N-1)v(e_1, \cdot). 
\end{align}
To see (\ref{eequivla-ND}), we  differentiate \eqref{eq:approxim} to get, for fixed $h \in 
C^{2,\alpha}_{p, e}(\R^m)$ 
\begin{align}
  0 = \frac{d}{ds}\Bigl|_{s=0} \Bigl(L^{1+ sh}_\l (u_n^{sh})\Bigr) 
    &= \Bigl(\frac{d}{ds}\Bigl|_{s=0} L^{1+ sh}_{\l, n}\Bigr)u_n + L^1_{\l, n} \frac{d}{ds}\Bigl|_{s=0}u_n^{sh} \nonumber\\
  &= \Bigl(\frac{d}{ds}\Bigl|_{s=0} L^{1+ sh}_{\l, n}\Bigr)u_n -  L_{\l, n} w_h  \label{zero-diff}  
\end{align}
with $w_h(t,x)=  \n u_n \cdot \t h(x)$, where we used \eqref{eq:approxisolu} in the last step. 
By the chain rule, we now have
\begin{equation}
D(H_\lambda \circ M)(0)v=\de_u H_\l(0,0)M_1 v+ \de_h H_\l(0,0)h_v \qquad \text{for $v \in X_2$,}
\label{eq:DGl0-compt}
\end{equation}
where, since by definition $M_1 v = v - w_{h_v}$ with $w_{h_v}(t,x) =  \n u_n \cdot \t  h_v(x)$,  
$$
\de_u H_\l(0,0)M_1 v = L_{\l, n} M_1 v =L_{\l, n} v - L_{\l, n}w_{h_v}
$$
and, by (\ref{zero-diff}), 
$$
\de_h H_\l(0,0)h_v = \Bigl(\frac{d}{ds}\Bigl|_{s=0} L^{1+ sh_v}_{\l, n}\Bigr)u_n = L_{\l, n}w_{h_v}.
$$
These identities  together with  \eqref{eq:DGl0-compt}  give  $DG_\l(0)v=L_{\l, n} v$ for $v \in X_2$ as desired. 
\QED\end{proof}
\vspace*{2mm} 
\section{Analysis of the linearized operator $D G_\lambda(0)$}\label{analyis}
In this section, we analyse the operator  $D G_\lambda(0)$  given in Proposition \ref{eqlinearised} and determined its kernel  as well as the image. To proceed, we first study the solutions of the following equation
\be\label{eq:w-nuDir}
- \de_{rr}w-\frac{N-1}{r}\de_r w=\nu w, \qquad -\de_r w(1)=(N-1) w(1),
\ee
with 
$
\nu\in \R. 
$
Here  \eqref{eq:w-nuDir} is equivalent to 
\be\label{eq:wvp}
\int_{0}^1  w'(r) \vp'(r)  r^{N-1} dr  +(N-1)w(1)\vp(1) =\nu \int_{0}^1  w(r) \vp(r)  r^{N-1} dr \qquad\textrm{ for all $\vp\in  C^{1}(0,1)$},
\ee
and we deduce 
 $$\int_{0}^1((w'(r))^2-\nu (w(r))^2)r^{N-1}dr=-(N-1)w(1)^2.$$  Hence  for $\nu\leq 0$ the only bounded solutions to  \eqref{eq:w-nuDir} is $w\equiv 0$. 

In the case where $\nu>0$, we have a family of solution to the interior equation  in \eqref{eq:w-nuDir} given 
\be\label{eq:ssoluUkk}
v(r)=A \psi_\nu(r), \quad \textrm{ for some $A \in \R$.}
\ee
where  $\psi_\nu(r)=I_{N/2-1}(r\sqrt{\nu })$.

Using the boundary conditions and provided $A \ne 0$, we  see that $\nu$ must solve
\be\label{eq:satisfsolv-eigen}
\frac{\sqrt\nu J_{\b+1}(\sqrt\nu)}{J_\b(\sqrt\nu)}=2\b+ 1 \qquad\textrm{ with $\b=\frac{N}{2}-1$}. 
\ee
For $N=1$, we have $2\b+ 1=0$ and  $\frac{J_{\b+1}(x)}{J_\b(x)}=\tan x$. Hence the solutions to \eqref{eq:satisfsolv-eigen} are given by $\nu_n=n^2\pi^2$. Recall in this case that   $j_{\b,n}=\frac{2n+1}{2}\pi$.\\

We now consider the case $N\geq 2$.  It is well  known that the map  $ x\mapsto \frac{J_{\b+1}(x)}{J_\b(x)}$ is increasing, has singularities at $j_{\b,n}$, negative on the intervals $(j_{\b,n},j_{\b+1,n})$ and  positive on the interval $(j_{\b+1,n},j_{\b,n+1})$. Moreover at $x=0$, it is equal to $0 $. It follows that its graph intersects the graph of the  convex and decreasing function $x\mapsto\frac{2\b+1}{x}$,   at the values $\nu_\ell$ in the order 
\be \label{eq:valuesordered}
0<\sqrt\nu_1<\sqrt\nu_2< \dots, \quad\textrm{ with $j_{\b+1,n}<\sqrt\nu_n<j_{\b,n+1}$ for $n\geq 1$, $\sqrt\nu_1< j_{\b,1}$.} 
\ee
We also recall,   $$
 j_{\nu,n}<j_{\nu+1,n}<j_{\nu, n+1} \qquad \text{for $\nu>-1$, $n\geq 1$} 
$$
(see e.g. \cite[Chapter XV, 15.22]{NevWatson})
and  from \cite[page 68, (1.5)]{A.Elbert}
\be\label{eq:growth-eigen}
\lim_{n\to\infty}  \frac{ j_{\nu,n}}{n}=\pi.
\ee
It is not difficult to check that $\psi_{\nu_n}$ form an orthogonal basis of eigenfunctions  in $L^2((0,1), r^{N-1})$. Indeed, using  \eqref{eq:wvp} we have 
\begin{align*}
 &\int_{0}^1  \psi'_{\nu_m}(r) \psi'_{\nu_n}(r) r^{N-1}dr   +(N-1) \psi_{\nu_m}(1)\psi_{\nu_n}(1)\\
  &=\nu_m \int_{0}^1  \psi_{\nu_m}(r) \psi_{\nu_n}(r) r^{N-1}dr=\nu_n \int_{0}^1  \psi_{\nu_m}(r) \psi_{\nu_n}(r) r^{N-1}dr.
\end{align*}
Hence either $\nu_n=\nu_m$ or  $\int_{0}^1  \psi_{\nu_m}(r) \psi_{\nu_n}(r) r^{N-1}dr=0$.\\

In the following, we let $\cP \subset \cL(\R^m)$ denote the subset
of all coordinate permutations and define the spaces
\begin{align*}
&X^k_{\mathcal{P}}:= \{ u \in  X_k\::\: u(\cdot , x)=  u(\cdot, \bold{p}(x))  \text{ for all $ x \in \R^N$, $\bold{p} \in \mathcal{P} $}\},\\
&Z^1_{\mathcal{P}}:= \{ h\in   Z_1 \::\: \quad h(x)=  h(\bold{p}(x))  \text{ for all $ x \in \R^N$ $  \bold{p} \in
\mathcal{P} $}\}.
\end{align*}
We claim that $G_\lambda$ sends  $ X^2_{\mathcal{P}}  \rightarrow  X^0_{\mathcal{P}}\times Z^1_{\mathcal{P}}$. \\

Indeed, we  observe that for a function $w\in C^{2, \alpha}(\R^m)$  satisfying  $w(x)=  w(\bold{p}(x))$   for all $ x \in \R^m$ $  \bold{p} \in
\mathcal{P} $,  if we put   $\bold{p}(x_i)=x_\ell$ and write   $y_j=\bold{p}(x_j) \in \{ x_1, \cdots x_m\}\setminus \{x_\ell\}$ for $j\ne i$.  Then 
\begin{align*}
\frac{\partial w}{\partial x_i}(\bold{p}(x_1), \cdots,  \bold{p}(x_m))&=\lim_{\e\rightarrow 0}\frac{ w(y_1, \cdots, y_{i-1}, y_i+\e, y_{i+1}, \cdots, y_m)-w(x_1, \cdots,  x_m)}{\e} \nonumber\\
&=\lim_{\e\rightarrow 0}\frac{ w(z_1, \cdots, z_{i-1}, z^\e_i, z_{i+1}, \cdots, z_m)-w(x_1, \cdots,  x_m)}{\e},  
\end{align*}
where  $z^\e_i=x_\ell+\e$  and $z_j=y_j=\bold{p}(x_j) \in \{ x_1, \cdots x_m\} \setminus\{x_\ell\}$ for $j\ne i$. We have  $\{ z_1, \cdots,  z^\e_i, \cdots, z_m \}= \{ x_1, \cdots, x_m, x_\ell+\e\}\setminus\{x_\ell\}$.  Then considering  the permutation $\sigma$ defined by  $\sigma(z_\ell)=x_\ell+ \e $,  and  $\sigma(z_k)=x_k$, $k\ne \ell $, it follows  from the property $w(z_1, \cdots   z_m)=w(\sigma(z_1), \cdots   \sigma(z_m))= w(x_1, \cdots,  x_\ell+\e, \cdots, x_m)$ that 
\begin{align}\label{eq:derivepermutrule}
\frac{\partial w}{\partial x_i}(\bold{p}(x_1), \cdots,  \bold{p}(x_m))& =\frac{\partial w}{\partial x_\ell}(x_1, \cdots,  x_m)
\end{align}
Thus   \eqref{eq:derivepermutrule}  together with  \eqref{eq:reldiffope-ND-Dir}, \eqref{eq: defQ} and  \eqref{eq:perturbed-strip-ND-Direqui} allow  to see that $G_\lambda$ sends $ X^2_{\mathcal{P}}  \rightarrow  X^0_{\mathcal{P}}\times Z^1_{\mathcal{P}}$. 

Next  for $j \in \N \cup \{0\}$, we define the Sobolev spaces
$$
H^{j}_{p,e}(\O_*):= \Bigl \{v \in H^{j}_{loc}(\O_*)\::\: \text{$v$ even,
$2\pi$-periodic in $x_1,\dots,x_m$}\Bigl \},
$$
$$
H^{j}_{p, rad }(\O_*) := \Bigl \{u \in  H^{j}_{p,e}(\O_*) \::\: \textrm{u is radial in $\t$} \Bigl \},
$$
$$
H^{j}_{\mathcal{P}, rad }(\O_*) := \Bigl \{u \in H^{j}_{p, rad }(\O_*) \::\: \textrm{ $u(\cdot , x)=  u(\cdot, \bold{p}(x))$} \text{ for all $ x \in \R^N$, $\bold{p} \in \mathcal{P} $}\}
$$
and set 
$$
\cH^{j}_{\mathcal{P},rad}(  \Omega_*):= \{ u \in H^{j}_{\mathcal{P}, rad }(\O_*) \::\: \text{$\n u\cdot \t+(N-1)u=0$ on $\de \O_*$  } \}.
$$
We also consider
\begin{equation}\label{eq:def-hpes}
H^{j}_{\mathcal{P}}(\R^m) := \Bigl \{\omega \in  H^{j}_{p,e}(\R^m)  \::\: \omega(x)=  \omega(\bold{p}(x))  \text{ for all $ x \in \R^N$, $\bold{p} \in \mathcal{P} $}\}, \qquad j \in \N \cup \{0\},
\end{equation}
where 
$$
H^{j}_{p,e}(\R^m):= \Bigl \{v \in H^{j}_{loc}(\R^m) \::\: \text{$v$ even,
$2\pi$-periodic in $x_1,\dots,x_m$}\Bigl \}.
$$
We also put  $L^2_{p,e}(\R^m):= H^{0}_{\mathcal{P}}(\R^m)$. Then  $L^2_{p,e}(\R^m)$ is a
Hilbert space with scalar product
$$
(u,v) \mapsto \langle u,v \rangle_{L^2} := \int_{[0,2\pi]^m}
u(t)v(t)\,dt \qquad \text{for $u,v \in L^2_{p,e}$.}
$$
We denote the induced norm by $\|\cdot\|_{L^2}$ and  define 
\be\label{decbasis}
\omega_{ k}(x)=\sum_{j=1}^m \cos( k x_j), \quad  k \in  \N \cup
\{0\}. 
\ee
Then the  family  $\widetilde{\omega}_k:=\frac{\omega_{k}}{\|\omega_k\|_{L^2}}$ forms an
orthonormal basis for $L^2_{p,e}(\R^m).$  

We set 
\be \label{eq:ei}
 \l_n:= {j_{N/2-1,n}^2-\nu_1 } 
\ee
and 
\be 
L_{\l_n,n}: =\D_t +\l_n \Delta_{x} +j_{N/2-1,n}^2 id.
\ee
With this  there holds
\begin{Lemma}\label{eq:lemmkerimG}
\begin{itemize}
\item[(i)]
Let $k,\ell \in \N$ be such that  $j_{\b,n}^2-\l_n k^2=\nu_\ell$. Then $k=1$ and $\ell=1$.
\item[(ii)]
Moreover, the non-trivial solution $v\in \cH^{2}_{\mathcal{P},rad}(  \Omega_*)$ to $L_{\l_n,n} v=0$ in $\O_*$  is given by 
$v(t)=A v_*$, for some $A\in \R^*_+$, where 
\be
v_*(t,x):=\psi_{\nu_1}(|t|)\vartheta(x), 
\ee 
with 
$$\vartheta(x):=\cos(x_1)+\cdots+\cos(x_m).$$ 
\item[(iii)]
The image of the linear map $DG_{\l_{n}}(0): X^2_{\mathcal{P}}  \rightarrow  X^0_{\mathcal{P}}\times Z^1_{\mathcal{P}}$ is given by 
\be\label{eq:inclusion}
Im\Big(DG_{\l_{n}}(0)\Big)= E_{\nu_1}^{\perp},
\ee
where 
\be\label{cokernel2}
E_{\nu_1}^{\perp} := \left\lbrace (w, h)\in  X^0_{\mathcal{P}}\times Z^1_{\mathcal{P}}: \int_{\Omega_*}   w (t, x)   v_*(t, x) \,dxdt- \int_{ \partial \Omega_*} h(x) v_*(t, x) dt dx=0 \right\rbrace.
\ee
\end{itemize}

\end{Lemma}
\begin{proof}
(i) Let $k,\ell \in \N$ be such that  $j_{\b,n}^2-\l_n k^2=\nu_\ell $. Then  $k\not=0$ because  non of the $\sqrt{\nu_\ell}$  given by  \eqref{eq:valuesordered} is a zero of $J_\b$. Furthermore, if  $j_{\b,n}^2-\l_n k^2=\nu_\ell$ then
\be
k^2=\frac{j_{\b,n}^2-\nu_\ell}{j_{\b,n}^2-\nu_1} \leq \frac{j_{\b,n}^2-\nu_1 }{j_{\b,n}^2-\nu_1} =1.
\ee
Hence $k=1$ and   $\ell=1$. 

(ii)
Write $v(t)=\sum_k v_k(|t|)\widetilde{\omega}_k(x)$.  Recalling 
\begin{align}\label{kerneleq}
D G_\lambda(0)v= \Bigl(j_{N/2-1,n}^2 v+ \lambda \Delta_{x} v + \D_\t v,  \n v(e_1, \cdot )\cdot e_1+(N-1)v(e_1, \cdot)\Bigl), 
\end{align}
The  equation  $\cL_n v=0$  with  $v\in \cH^{2}_{\mathcal{P},rad}(  \Omega_*)$ 
implies that the coefficients  $v_k$ solve  \eqref{eq:w-nuDir} and  from \eqref{eq:ssoluUkk} \eqref{eq:satisfsolv-eigen}, $v_k\ne0$ if and only if 
\be
j_{\b,n}^2-\l_n k^2=\nu_\ell, \qquad\textrm{ for some $\ell\in \N$.} 
\ee
Then (i) implies that this is possible only when  $k=1$ and $\ell=1$.  We thus get $v_k\equiv 0$ for all $k\not=1$ and  $v_1$ is clearly proportional  to the eigenfunction $\psi_{\nu_1}$.

(iii)
Let  $(w,h)\in  Im\Big(DG_{\l_n}(0)\Big) \subseteq  X^0_{\mathcal{P}}\times Z^1_{\mathcal{P}}.$  Then there exists $U \in  X^2_{\mathcal{P}}$ such that 
\begin{equation}\label{eq:diffeq1}
DG_{\l_n}(0)U= (w,h).
\end{equation} 
That is equivalent to 
\begin{equation}\label{eq:diffeq33}
DG_{\l_n}(0)U = (w,h)\Longleftrightarrow \left\{\begin{aligned}
\D_{\t}U + \lambda_n \partial_{xx}U +  j_{\b,n}^2 U&=w \quad\text{in}\quad \O_*\\
\n U(e_1,x)\cdot e_1+(N-1)U &= h\quad\text{on}\quad \partial\O_*. 
 \end{aligned}
\right.
\end{equation}
We define $\cB: H^{1}_{p, rad }(\O_*) \times H^{1}_{p, rad }(\O_*)\to \R$ 
\be\label{bilieanrform}
\cB(u,v) =\int_{\O_*} [\n_t u \cdot \n_t v+\l_n\de_x   u \de_x v]-   j_{\b,n}^2\int_{\O_*} uv+(N-1)\int_{\de \O_*} uv.
\ee
Multiply  \eqref{eq:diffeq33} by $\vp\in C^1(\ov\O_*) $ and integrate by parts to have 
$$
\cB(U,\vp)=-\int_{\O_*}w\vp +\int_{\de\O_*} h\vp 
$$
It is clear that  $\cB(U,v_*)=0$ and  we  immediately  deduce   $-\int_{\O_*}w v_* +\int_{\de\O_*} h v_*=0 $, so that 
\be
Im\Big(DG_{\l_{n}}(0)\Big) \subseteq  E_{\nu_1}^{\perp}. 
\ee 
We prove next the other inclusion. 
Set $\psi_{n}:=\frac{\psi_{\nu_n}}{\|\psi_{\nu_n}\|_{L^2(B_1)}}$, where $\nu_m$ are given by \eqref{eq:valuesordered} and  $\psi_{\nu_m}$ solve \eqref{eq:w-nuDir} together  with  \eqref{eq:satisfsolv-eigen}. 
Writing, 
\begin{equation}\label{eqdecomp00}
u(t,x)=\sum_{k,\ell\in \N } u_{k,\ell}  \psi_{\ell}(|t|)\widetilde{\omega}_k(x), 
\end{equation}
 the norm 
$\|u\|_{H^{1}_{p, rad }(\O_*)}$ is equivalent to $\sum_{k,\ell}(1+\ell^2+k^2)u_{\ell,m}^2$. 
Furthermore, it is clear from  \eqref{bilieanrform} that 
\begin{align*}
\cB(u,u)\geq  \sum_{ k, \ell\in \N  } \s_{k,\ell}u_{k,\ell}^2 
\end{align*}
where 
\be \label{eq:sigma}
{\s}_{k,\ell}:= \nu_\ell- j_{\b,n}^2+k^2\l_n. 
\ee
Note that the set of $(\ell,k)\in \N\times \N$ such that  $\s_{k,\ell}\leq 0$ is finite. 
In addition   \eqref{eq:valuesordered}  and \eqref{eq:growth-eigen} show that   $\s_{k, \ell}\backsim \ell^2+k^2$ as $\ell^2+k^2 \rightarrow \infty$.  

We then get  positive constants $C,c>0$ such that    
\begin{align}\label{near-coerciv}
\cB(u,u) \geq  C \|u\|_{H^{1}_{p, rad }(\O_*)}^2- c\|u\|_{L^2(\O_*)}^2,  \forall u\in H^{1}_{p, rad }(\O_*),
\end{align}

Note that by the compact embedding of  $H^{1}_{p, rad }(\O_*)$ into $L^2(\O_*)$ , the symmetric bilinear operator 
\be
B(u,v) =\int_{\O_*} vu 
\ee
is compact on $H^{1}_{p, rad }(\O_*)$. In addition  by \eqref{near-coerciv},  for large $\delta>0$, the bounded symmetric  bilinear form $\cB+\delta B:P H^{1}_{p, rad }(\O_*)\times P  H^{1}_{p, rad }(\O_*)\to \R$ is strictly positive definite and thus strictly nondegenerate by the Lax-Milgram thorem, where $P:L^2(\O_*)\to L^2(\O_*)$ denotes the $L^2(\O_*)$-orthogonal projection on $\la v_*\ra^\perp$.  It follows that  $\cB: P H^1_{p, rad }(\O_*)\times P H^1_{p, rad }(\O_*)\to \R$ satisfies the Fredholm's alternative, see \cite[Theorem 2.3]{Kress}. Therefore,
letting  $\ell\in (H^{1}_{p, rad }(\O_*))'$ be given  by 
$$
\ell(\vp):= -\int_{\O_*} w\vp+\int_{\de\O_*} \vp h
$$  
then either
\begin{enumerate}
\item   there exists a unique  $v\in  P H^{1}_{p, rad }(\O_*)$ such that  $ \cB(v, \cdot )=\ell(\cdot)$ 
\item  or the equation $\cB(v,\cdot)=0$ admits  a nontrivial solution  $v\in P H^{1}_{p, rad }(\O_*)$. 
\end{enumerate}
 Since  $(ii)$ is impossible because   $\cB (v,\cdot)=0$ if and only if $v=A v_*$ for some $A\in \R$, we then have $(i)$:  there exists a unique  $v\in  P H^{1}_{p, rad }(\O_*)$ such that 
 $$
 \cB(v, \vp)=\ell(\vp)\qquad\textrm{ for all $ \vp \in P  H^{1}_{p, rad }(\O_*))$}
 $$
 Now for $\vp\in H^{1}_{p, rad }(\O_*)$, we have  $\vp=P\vp+(id-P) \vp=P\vp+\t v_*$, for some $\t\in \R$.  Since $(w,h)\in E_{\nu_1}^{\perp}$, we have $\ell(v_*)=0$ and thus, recalling that $\cB(v_*,\cdot)=0$, we obtain  
 $$
 \cB(v, \vp)=\ell(\vp)\qquad\textrm{ for all $\vp\in H^{1}_{p, rad }(\O_*)$.} 
 $$
 Now by elliptic regularity theory, we have that  $v\in C^{2,\a}(\ov\O_*)$ as soon as  $(w,h)\in C^{0,\a}(\ov\O_*)\times C^{1,\a}(\partial \O_*).$ Furthermore by uniqueness, it follows from  \eqref{eq:diffeq33} that  
$v\in H^{2}_{\mathcal{P}, rad }(\O_*)$ since   $(w,h)\in H^{0}_{\mathcal{P}, rad }(\O_*)\times \in H^{1}_{\mathcal{P} }(\R^m).$ Consequently, $v\in  X^2_{\mathcal{P}} = H^{2}_{\mathcal{P}}(\O_*)  \cap C^{2,\alpha}(\overline \Omega_*)$ and
\be
E_{\nu_1}^{\perp} \subseteq Im\Big(DG_{\l_{n}}(0)\Big), 
\ee
as desired.
\QED
\end{proof}

We   can  now summarise the previous analysis as follows. 
\begin{Proposition}\label{propCR-ND-Dir}
We have the following properties.
\begin{itemize}
\item[(i)] The kernel $N(\cL _n)$ of $\cL _n: =D G_{\lambda_n}(0): X^2_{\mathcal{P}}  \rightarrow  X^0_{\mathcal{P}}\times Z^1_{\mathcal{P}}$  is spanned  by $v_*(t,x)= \psi_{\nu_1}(|t|)\vartheta(x)$,
with 
$$\vartheta(x):=\cos(x_1)+\cdots+\cos(x_m).$$ 
\item[(ii)] The range of $\cL_n $ is given by
$$
R(\cL_n )= E_{\nu_1}^{\perp}.
$$
\item[(iii)]  Moreover,
\begin{equation}
  \label{eq:transversality-cond4-ND-Dir}
\partial_\lambda \Bigl|_{\lambda=\l_n }G_{\l}(0)(v_*)\not  \in \; R(\cL _n).\\
\end{equation}
\end{itemize}
\end{Proposition}

\begin{proof}
$ (i)$   and (ii)  obviously follow from Lemma \ref{eq:lemmkerimG}. To get (iii), we use \eqref{eequivla-ND}  and find 
$$
\de_\l\big|_{\l=\l_{n}} DG_{\l }(0) v_* = (\Delta_{x}v_*, 0)=(-v_*, 0).
$$
The  proof is complete.
   \QED
\end{proof}

\section{Proof of Theorem \ref{Theo1-ND}  }\label{eq:ProofTheo1-ND}
The proof of Theorem  \ref{Theo1-ND}   is achieved  by applying 
the Crandall-Rabinowitz Bifurcation theorem to solve the equation
\begin{align}\label{eq:maptsolvG}
 G_\lambda(u)=0,
\end{align}
where   $G_{\l}: \cU  \cap X^2_{\mathcal{P}}  \rightarrow  X^0_{\mathcal{P}}\times Z^1_{\mathcal{P}}$  is  defined by 
\eqref{eq:perturbed-strip-ND-Direqui}. 

\begin{Theorem}\label{Theo1-general-ND}
For every $n\in \N$, there  exist  ${\e_n}>0$ and a smooth curve
$$
(-{\e_n},{\e_n}) \to   (0,+\infty) \times  X^2_{\mathcal{P}},\qquad s \mapsto (\lambda_n(s),\varphi^n_s)
$$
with
$\lambda_n(0)= \l_{n}$, $ \varphi_n(0)\equiv 0$ such that  
\be
G_{\l_n(s)}(\varphi_n(s)) =0.
\ee
Moreover, $\varphi^n_s = s(v_{*}+\o_n(s)),$ with a smooth curve 
$$
(-{\e_n},{\e_n}) \to   X^2_{\mathcal{P}}, \qquad s \mapsto \o_n(s)
$$ satisfying  $ \o_n(0) =0$ and 
\begin{align*}
\int_{\O_*}  \o_n(s) (\t,x)  v_{*}(\t,x)\,dxd\t=0,
\end{align*}
where
$$v_{*} (\t,x)=  I_{N/2-1} (|\t| \sqrt{\nu_1} )(\cos(x_1)+\cdots+\cos(x_m)).$$   
In addition, setting
 \begin{equation}\label{solutinter}
u_s(\t, x):=  \varphi^n_s(\t,x) -\frac{|\t|}{I_{N/2-1}'(j_{N/2-1,n})} I_{N/2-1}'(j_{N/2-1,n}|\t|)  \varphi^n_s(e_1,x)  + u_n(\t, x)
 \end{equation}
and  
 \begin{equation}\label{eexpreh}
 h _{\varphi^n_s}(x):= \frac{1}{ j_{N/2-1,n} I_{N/2-1}'(j_{N/2-1,n})}    \varphi^n_s(e_1,x), 
 \end{equation}
 the function
\begin{align}\label{eqans2-ND33}
U_s(t,x)=u_s((1+h _{\varphi^n_s})t,x)
\end{align} satisfies 
\begin{equation}\label{eq:solved-ND1}
  \left \{
    \begin{aligned}
         \lambda_n(s) \de_{xx} U _s +\D_t  U _s+  j^2_{N/2-1,n} U _s &=0&& \qquad \text{in $ \Omega_{  {1+h _{\vp^n_s}} }$,}\\
             U_s&=0 &&\qquad \text{on $\partial \Omega_{ 1+h _{\vp^n_s}}$,}\\
              |\n U _s| &=j_{N/2-1,n} I'_{N/2-1}(j_{N/2-1,n})  &&\qquad \text{on $\partial  \Omega_{ 1+ h _{\vp^n_s}}$}.
    \end{aligned}
       \right.
\end{equation}
\end{Theorem}

\begin{proof}
 We consider the smooth  map  $G_{\l_n}: \cU  \cap X^2_{\mathcal{P}}  \rightarrow  X^0_{\mathcal{P}}\times Z^1_{\mathcal{P}}$  and define 
\begin{align}\label{eq:spaorthokernel}
\mathcal{X}^{\perp} := \left\lbrace  v \in  X^2_{\mathcal{P}}: \int_{\O_* } v(t,x)v_{*}(t,x)\,dxdt =0\right\rbrace.
\end{align}
By Proposition \ref{propCR-ND-Dir}  and  the Crandall-Rabinowitz Theorem (see \cite[Theorem 1.7]{M.CR}), we then find ${\e_n}>0$ and a smooth curve
$$
(-{\e_n },{\e_n }) \to  (0,\infty) \times \mathcal{U} \subset  \R_+ \times  X^2_{\mathcal{P}}, \qquad s \mapsto (\lambda_n(s), \varphi^n_s)
$$
such that
\begin{enumerate}
\item[(i)] $G_{\lambda_n(s)} (\varphi^n_s)=0$ for $s \in (-{\e_n },{\e_n })$,
\item[(ii)] $ \lambda_n (0)= \l_{n}$, and
\item[(iii)]  $\varphi^n_s = s v_{*}+ s \o_n(s) $ for $s \in (-{\e_n },{\e_n })$ with a smooth curve
$$
(-{\e_n },{\e_n }) \to \mathcal{X} ^{\perp}, \qquad s \mapsto \o_n(s)
$$
satisfying $ \o_n(0) =0$
and
$$\int_{\O_* } \o_n(s) (t,x) v_{*}(t,x)\,dxdt=0.$$
\end{enumerate}
Recalling   \eqref{eq:solfinal},  since  $G_{\lambda(s)} (\varphi^n_s)=0$ for every  $s \in (-{\e_n },{\e_n })$, using $\n u_n(\t, x) \cdot \t = |t|j_{N/2-1,n} I'_{N/2-1}(j_{N/2-1,n}|t|)$, we have from \eqref{eq:solfinal}  and  \eqref{eq:defM} that the function 
$$
u_s:=  \varphi^n_s-h_{\varphi^n_s} \n u_n \cdot \t+u_n
$$ in \eqref{solutinter} solves   \eqref{eq:perturbed-strip-ND-Direqui1} in $\O_*$,  with $h _{\varphi^n_s}(x) = \frac{1}{ j_{N/2-1,n} I_{N/2-1}'(j_{N/2-1,n})} \varphi^n_s(e_1,x) $.  Finally, we deduce from    \eqref{eqans1} that the  solution of   \eqref{eq:solved-ND1} is  given by \eqref{eqans2-ND33}. 
\QED
\end{proof}

\subsection*{Proof of Theorem \ref{Theo1-ND} (completed)}
In view of Theorem \ref{Theo1-general-ND}  and \eqref{eq:changsoluD},  the function
\begin{equation}\label{Solfinalw}
(t,x )\mapsto w_s(t,x)=  U_s( t/ \sqrt{\l_n(s)},x)= u_s((1+h _{\varphi^n_s})/\sqrt{\l_n(s)}t,x)
 \end{equation}
 solves \eqref{eq:solved-main-ND} with $\mu^n_s=\frac{1}{\l_n(s)} j^2_{N/2-1,n}$
 $$
 h^n_s(x)=\frac{1+h_{\varphi^n_s}(x)}{\sqrt{\l_n(s)}}.  
 $$
Furthermore using  \eqref{Neumanndata} and \eqref{constc1},  the  Neumann boundary data in \eqref{eq:solved-main-ND} is  given by 
$$c=\frac{j_{N/2-1,n}  I'_{N/2-1}(j_{N/2-1,n})}{\sqrt{\l_n(s)}}=I'_{N/2-1}(j_{N/2-1,n})\sqrt{\mu^n_s}.$$
Recalling  $$\l(n)=j_{N/2-1,n}^2-\nu_1 \quad \textrm{and} \quad v_{*}(t,x)=I_{N/2-1} (|t| \sqrt{\nu_1})\vartheta(x),$$ where $\vartheta(x):=\cos(x_1)+\cdots+\cos(x_m)$ from  Theorem \eqref{Theo1-general-ND}, we have that 
$$\varphi^n_s(t, x) = s I_{N/2-1} (|t| \sqrt{\nu_1})\vartheta(x) + o(s),$$ where $o(s) \to 0$ in $C^2$-sense in $\ov{\Omega_*}$ as $s \to 0$. 
Hence using \eqref{eexpreh}, 
\begin{equation}\label{eexpre2}
 h _{\varphi^n_s}(x):= s\frac{ I_{N/2-1} ( \sqrt{\nu_1})}{ j_{N/2-1,n} I_{N/2-1}'(j_{N/2-1,n})}\vartheta(x) + o(s) \nonumber\\, 
 \end{equation}
and therefore, 
\begin{align*}
  h^n_{s}(x)&=\frac{1+h_{\varphi^n_s}}{\sqrt{\l_n(s)}}= \frac{1}{\sqrt{\l_n(s)}} + s\frac{I_{N/2-1} ( \sqrt{\nu_1})}{\sqrt{\l_n(s)}j_{N/2-1,n} I_{N/2-1}'(j_{N/2-1,n})}\vartheta(x)+ o(s)\\
&=  \frac{1}{j_{N/2-1,n}} \sqrt{\mu^n_s} + s\frac{I_{N/2-1} ( \sqrt{\nu_1}) }{j_{N/2-1,n} I_{N/2-1}'(j_{N/2-1,n})\sqrt{j_{N/2-1,n}^2-\nu_1}}\vartheta(x)+ o(s) \qquad \text{as $s \to 0$}.
\end{align*}
Finally, by \eqref{Solfinalw}  and  \eqref{solutinter}, 
\begin{align}\label{eqexpannfi}
&w_s(\frac{t}{h^n_s(x))},x)= u_s(\t, x)=  u_n(\t, x)+ \varphi^n_s(\t,x) -\frac{|\t|}{I_{N/2-1}'(j_{N/2-1,n})} I_{N/2-1}'(j_{N/2-1,n}|\t|)  \varphi^n_s(e_1,x)  \\ 
  &=U_n(|t|) + s \Bigl( I_{N/2-1} (|t| \sqrt{\nu_1})-\frac{I_{N/2-1} ( \sqrt{\nu_1}) }{ j_{N/2-1,n} I_{N/2-1}'(j_{N/2-1,n})}|t|U_n'(|t|) \Bigr)\vartheta(x) + o(s),\nonumber
\end{align}
where $o(s) \to 0$ in $C^1$-sense on $\Omega_*$.
We thus have proved Theorem~\ref{Theo1-ND} with the constants
\begin{align*}
\mu_n& =\frac{j^2_{N/2-1,n}}{\l_n(0)} =\frac{j^2_{N/2-1,n}}{\l_n}= \frac{j^2_{N/2-1,n}}{j^2_{N/2-1,n} - \nu_1},\qquad \quad   \kappa_n = \frac{1}{j_{N/2-1,n}},\\
\beta_n &=\frac{I_{N/2-1} ( \sqrt{\nu_1}) }{j_{N/2-1,n} I_{N/2-1}'(j_{N/2-1,n})\sqrt{j_{N/2-1,n}^2-\nu_1}}\,\\
\gamma_n&=-\frac{I_{N/2-1} ( \sqrt{\nu_1}) }{ j_{N/2-1,n} I_{N/2-1}'(j_{N/2-1,n})}
\end{align*}
and with the function
$$
t \mapsto \psi_1(|t|)= I_{N/2-1} (|t| \sqrt{\nu_1}) \quad \textrm{and}\quad U_n(|t|)=u_n(t, x)=I_{N/2-1}(j_{N/2-1,n}|\t|).
$$
\QED

\begin{Remark}\label{noncostNeum}
In this  remark  we  discuss  how  the  approach of  the previous sections can be applied to solve  an overdetermined  problem  with non constant Neumann boundary data related  to  problem  \eqref{eq:Dirichlesign}. 
Using  \eqref{eq:derfin},  we can write 
$$
  \frac{\partial w^\lambda }{\partial\eta_{ \frac{h}{ \sqrt{\lambda}} }}( \frac{ \sqrt{\lambda }}{h(x)} e_1, x)  =  \frac{1}{\sqrt{\lambda}}  \frac{\sqrt{1+ \lambda\frac{ |\nabla h(x)|^2 }{ h^4(x)} } }{\sqrt{1+ \frac{ |\nabla h(x)|^2 }{ h^4(x)} }}  \Biggl(\sqrt{1+ \frac{ |\nabla h(x)|^2 }{ h^4(x)} } \n_t v ( \frac{1}{h(x)} e_1, x)\cdot e_1 \Biggl).
$$
This  with \eqref{eq:changsoluD} allows  to  see that a radial function $v$ in $t$ solves 
\begin{equation}\label{eq:per3}
  \left \{
    \begin{aligned}
& L_{\l,n} v =  0 && \qquad \text{in $\Omega_{1+h}$,}\\
          & v =0 &&\qquad \text{on $\partial \Omega_{1+h}$,}\\
     & \sqrt{1+ \frac{ |\nabla h(x)|^2 }{ (1+h(x))^4} } \n_t v ( \frac{1}{1+h(x)} e_1, x)\cdot e_1 = \n u_n (e_1,\cdot) \cdot  e_1 && \qquad  \text{  $ x\in \R^m$,} \\
    \end{aligned}
       \right.
     \end{equation}
 if and only if  the function $w^\lambda$  solves  the problem  \eqref{eq:Dirichlesign}  in $\partial \Omega_{1+h}$ with   non constant  Neumann boundary data 
\begin{equation}\label{eq:Neumandaat2} 
 c(x):= \frac{\n u_n (e_1, x) \cdot  e_1}{\sqrt{\lambda}}  \frac{\sqrt{1+ \lambda\frac{ |\nabla h(x)|^2 }{ (1+h(x))^4} } }{\sqrt{1+ \frac{ |\nabla h(x)|^2 }{(1+h(x))^4} }}.
\end{equation}
We now emphasize that  problem \eqref{eq:per3} is solved in a similar fashion than problem \eqref{eq:per}. Indeed  as already stated   at the end of Section \ref{sectio1},  the problem  \eqref{eq:per} is equivalent to  \eqref{eq:perturbed-strip-ND-Direqui1}.  Similarly,  we see that \eqref{eq:per3}  is equivalent to
\begin{equation}
  \label{eq:perturbed-strip-ND-Direqui2}
  \left \{
    \begin{aligned}
&L_{\l,n}^{1+h} u =  0 && \qquad \text{in $\Omega_*$,}\\
          & u =0 &&\qquad \text{on $\partial \Omega_*$,}\\
     & (1+h)\sqrt{1+ \frac{|\nabla h(x)|^2 }{(1+h)^4}}   \n_\tau u(e_1, \cdot) \cdot e_1 =\n u_n (e_1,\cdot) \cdot  e_1. &&\qquad \text{in  $\R^m$.} 
    \end{aligned}
       \right.
\end{equation}
We also  note that  when $h=1$, both the problems  \eqref{eq:per} and \eqref{eq:per3} are solved by the $\lambda$-independent  function  
$
 u_n(t, x):= I_{N/2-1}(j_{N-1/2,n}\, |t|).
$
Now by defining 
\begin{align}\label{eq the mapGtilde}
\widetilde{G}_\lambda: \cU   \to X_0  \times Z_1, \qquad G_\lambda(u):= (H_\lambda \circ M(u),  Q_{1}\circ M(u)),
\end{align}
where 
\begin{align}\label{eq: defQ2}
 Q_{1}\circ M(u)&=( 1+ h_u) \sqrt{1+ \frac{|\nabla h_u(x)|^2}{ (1+h_u)^4}}\Bigl(\n u(e_1,\cdot) \cdot  e_1+\n u_n (e_1,\cdot) \cdot  e_1+(N-2)  u (e_1, \cdot)\Bigl)\nonumber\\
& -\n u_n (e_1,\cdot) \cdot  e_1,
\end{align}
it follows that 
\begin{align}\label{eq bifurcate2}
\widetilde{G}_\lambda(u)=0
\end{align}
if and only if   \eqref{eq:perturbed-strip-ND-Direqui2} is solved by  the function in \eqref{eq:solfinal}  with $h$ replaced by $1+h_u$, where  $h_u$ is defined  by \eqref{eq exprehh}.\\

Observe now that  $Q_{1}\circ M $  and $Q_{\lambda}\circ M $ in  \eqref{eq: defQ} have the same linearred operator at $u=0$. Hence  $\widetilde{G}_\lambda$ and $G_\lambda$ share  the same linearised operator in  Proposition \ref{eqlinearised}. Following step by step the argument in the previous sections, one solves \eqref{eq bifurcate2} in the same fashion than \eqref{eq bifurcate}. 
\end{Remark}
 
\section{On the  partially overdetermined problem  \eqref{eq:perturbed-partially} }\label{Parrtialy}
This section is devoted to the overdetermined boundary problem  \eqref{eq:perturbed-partially}.  As already emphasised in the introduction, our setting  will involve spaces of functions that are odd in the variable $t$ for  the condition   $\frac{\partial u}{\partial \eta}  = \pm 1\quad \text{on $\partial \Omega^{\pm}$}$ to hold. 

We consider the open set
$$
\cU_0:= \{ h \in  C^{2,\a}_{p, e}(\R^m) \::\: h >0 \}
$$
 and define for a function  $ h \in \cU_0$, the following domain
\begin{equation}\label{eq:PertTorus2}
\widetilde{\Omega}_h:= \left\{\left(t, x \right)\in  \R \times \R^m  \::\: |t|<\frac{1}{h(x)}  \right\}.
\end{equation}

We  are interested in the overdetermined  boundary  value problem
\begin{equation}\label{eq:pe5}
  \left \{
    \begin{aligned}
\D w +\mu  w &=  0 && \qquad \text{in $ \widetilde{\Omega}_h$,}\\
              u &=0 &&\qquad \text{on $\partial \widetilde{\Omega}_h$,}\\
             \frac{\partial w}{\partial \mu}  &= \pm \gamma  &&\qquad \text{on $\partial  \widetilde{\Omega}_h^{\pm}$,}
    \end{aligned}
       \right.
     \end{equation}
where  $\mu$ is the outer unit normal vector field to the boundary of  $\Omega_h$, $\gamma$ is a  positive constant   and 
$$ \partial \widetilde{\Omega}_h^{\pm}=  \left\{\left(\pm\frac{1}{h(x)}, x\right): \quad x \in \R^m \right\}.$$
For a fixed  integer $n\geq 1$, we  define the operator
\begin{equation}
  \label{eq:opLlam-Dir}
  \widetilde{L}_{\l,n}:= \lambda  \D_{x} + \partial_{tt}  +  n^2\pi^2 \textrm{id}   
 \end{equation}
 with  $\lambda >0$. Then as  in   \eqref{eq:changsoluD},
\begin{equation}
  \label{eq:changsoluD2}
w^{\lambda} (t, x)= v( t/\sqrt{\lambda}, x)
\end{equation}
solves  the  Dirichlet problem  in  \eqref{eq:pe5}  with $\mu=\frac{n^2 \pi^2 }{\lambda}$  on the domain  $\O_{\frac{h}{ \sqrt{\lambda}}}$ if and only if $v$ solves 
\begin{equation}\label{eq:perturbed-strip5}
  \left \{
    \begin{aligned}
 \widetilde{L}_{\l, n}  v &=  0 && \qquad \text{in $\Omega_h$,}\\
              v &=0 &&\qquad \text{on $\partial \Omega_h$.}
    \end{aligned}
       \right.
     \end{equation}
Furthermore the  outer normal on $\partial  \widetilde{\Omega}_h $ with respect to the Euclidean
metric $g_{eucl}$ given by
\begin{equation}
  \label{eq:def-mu-phi}
\mu_h (t, x) = \dfrac{ 1}{\sqrt{1+ \frac{ |\nabla h'(x)|^2 }{ h^4(x)} }} \Bigl(\dfrac{t}{|t|},  \dfrac{ \nabla h(x) }{ h^2(x)}, \Bigl)  \in \R\times \R^{m} \qquad \text{for
$(t, x) \in \partial  \widetilde{\Omega}_h$}
\end{equation}
 and we obtain
\begin{align} \label{eq:dlw}
  \frac{\partial w^\lambda }{\partial\mu_{ \frac{h}{ \sqrt{\lambda}} }}(t, x) = \dfrac{ 1}{\sqrt{1+ \lambda \frac{ |\nabla h(x)|^2 }{ h^4(x)} }}  \Bigl[ \dfrac{t}{|t|} \partial_t w^\lambda (t, x)   +  \sqrt{\lambda }\frac{ \n h(x) }{ h^2(x)} \cdot \n_x w^\lambda (t, x) \Bigl].
\end{align}
Since we require $w^{\lambda} (t, x)=0$ on  $\partial \Omega_{h/\sqrt{\lambda}}$, we have 
 $ w^{\lambda} ( \pm \frac{ \sqrt{\lambda }}{h(x)}, x)=0$ for all $x\in \R^m$ and differentiating  this with respect to $x$, 
$$
  \n_x w^\lambda(\pm \frac{ \sqrt{\lambda }}{ h(x)}, x)= \pm\sqrt{\lambda }\partial_t w^\lambda ( \pm\frac{ \sqrt{\lambda }}{h(x)}, x)\frac{ \n h(x) }{ h^2(x)} 
$$ 
so that 
\begin{align}\label{eqfin}
  \frac{\partial w^\lambda }{\partial \mu_{ \frac{h}{ \sqrt{\lambda}} }}( \pm\frac{ \sqrt{\lambda }}{h(x)}, x) &=\pm \sqrt{1+ \lambda \frac{ |\nabla h(x)|^2 }{ h^4(x)} } \partial_t w^\lambda  ( \pm\frac{ \sqrt{\lambda }}{h(x)}, x)=  \pm \frac{1}{\sqrt{\lambda}}\sqrt{1+ \lambda \frac{ |\nabla h(x)|^2 }{ h^4(x)} }  \partial_t v ( \pm\frac{1}{h(x)}, x).
\end{align}

\begin{Remark}\label{Rmm}
It is obvious that if  $ w^\lambda$ is odd in $t$, then  $\partial_t w^\lambda  ( -\frac{ \sqrt{\lambda }}{h(x)}, x)=\partial_t w^\lambda  ( +\frac{ \sqrt{\lambda }}{h(x)}, x)$  and the first equality in  \eqref{eqfin}  yields  $ \frac{\partial w^\lambda }{\partial \mu_{ \frac{h}{ \sqrt{\lambda}} }}( +\frac{ \sqrt{\lambda }}{h(x)}, x) =-\frac{\partial w^\lambda }{\partial \mu_{ \frac{h}{ \sqrt{\lambda}} }}(-\frac{ \sqrt{\lambda }}{h(x)}, x)$. Therefore  
 \begin{equation}\label{eq:normlllls}
 \frac{\partial w^\lambda }{\partial \mu_{ \frac{h}{ \sqrt{\lambda}} }}( +\frac{ \sqrt{\lambda }}{h(x)}, x)=+\frac{1}{\sqrt{\lambda}} \Longleftrightarrow   \frac{\partial w^\lambda }{\partial \mu_{ \frac{h}{ \sqrt{\lambda}} }}( -\frac{ \sqrt{\lambda }}{h(x)}, x)=-\frac{1}{\sqrt{\lambda}}. 
\end{equation}
We also   see  from the second  equality in  \eqref{eqfin} that 
 \begin{equation}\label{eq:equiNeuman}
 \frac{\partial w^\lambda }{\partial \mu_{ \frac{h}{ \sqrt{\lambda}} }}( +\frac{ \sqrt{\lambda }}{h(x)}, x)=+\frac{1}{\sqrt{\lambda}} \Longleftrightarrow  \sqrt{1+ \lambda \frac{ |\nabla h(x)|^2 }{ h^4(x)} }  \partial_t v ( \pm\frac{1}{h(x)}, x)=+1. \\
\end{equation}
 \end{Remark}
From  \eqref{eq:equiNeuman} and  \eqref{eq:changsoluD2}, we  have that an  odd function $v$  in the variable  $t$ solves 
\begin{equation}\label{eq:per4}
  \left \{
    \begin{aligned}
& \widetilde{L}_{\l,n} v =  0 && \qquad \text{in $\Omega_h$,}\\
          & v =0 &&\qquad \text{on $\partial \Omega_h$,}\\
     &\sqrt{1+ \lambda \frac{ |\nabla h(x)|^2 }{ h^4(x)} }  \partial_t v ( \pm\frac{1}{h(x)}, x)=+1 && \qquad  \text{  $ x\in \R^m$,} \\
    \end{aligned}
       \right.
     \end{equation}
if and only if the function  $w^\lambda$ in \eqref{eq:changsoluD2}  solves   the problem \eqref{eq:pe5} with  
\begin{equation}\label{eq:Neumandata22}
\gamma = \frac{1}{\sqrt{\lambda}}.
 \end{equation} 
We emphasize that when $h=1$ in  \eqref{eq:per4}, then  $\widetilde{\Omega}_1:=\Omega_1= (-1, 1)\times \R^m$ and  the problem  \eqref{eq:per4}  is solved by the $\lambda$ independent function 
\begin{equation}\label{eq:trivialsol1}
v_n(t, x):=\frac{(-1)^n}{n\pi} \sin(n\pi t), \quad t\in  (-1, 1). 
  \end{equation}

Observe  also  that  set  that  $\widetilde{\Omega}_h$ is parametrized by the mapping
$$ \Psi_h: \Omega_1  \to  \widetilde{\Omega}_h, \quad  (\t, x)  \mapsto (t, x)=(\frac{\t}{h(x)}, x ),$$ with inverse given by 
$ \Psi^{-1}_h:   \widetilde{\Omega}_h   \to  \Omega_1, \quad  (t, x)  \mapsto   (h(x) t, x)$.\\

We pull-back \eqref{eq:per4}  on the fixed domain   $\Omega_1=(-1, 1)\times \R^m$ via the ansatz 
\begin{align}\label{eqansnew}
v(t, x)= u(h(x)t, x)=u(\tau, x) \qquad \text{for some function $u: \Omega_1 \to \R$.}
\end{align}
We need  to find the differential operator $\widetilde{L}^h_\lambda$ with the property that 
\begin{align}\label{eqreladiffopets}
[\widetilde{L}^h_\lambda u] (h(x)t, x) = [\widetilde{L}_{\l,n}  v](t, x) \qquad \text{for $(t, x) \in \Omega_h$}
\end{align}
for the function $v: \Omega_h \to \R$, $v(t, x)=u(h(x)t, x)$.\\

A direct computation yields 
\begin{align*}
\widetilde{L}_\lambda^h v_h(\t,x) =&n^2\pi^2 u(h(x)t,x) + \lambda \Delta_{x}u(h(x)t,x) +   h^2(x) \partial^2_\tau u(h(x)t,x)   \\
                      &+ \lambda  t^2 |\nabla_x h(x)|^2  \partial^2_\tau u(h(x) t,x) + 2 \lambda  t \nabla_x h(x) \nabla_x \partial_\tau u(h(x) t,x)\\
  &+  \lambda  t \Delta h(x) \partial_t u(h(x) t,x) \qquad  \text{for $(t,x) \in \Omega_h$.}
\end{align*}
Recalling $\tau=h(x)t$,  therefore gives   
\begin{align} 
\widetilde{L}_\lambda^h v_h(\t,x) =&n^2\pi^2 u(\tau,x) + \lambda \Delta_{x} u(\tau, x) + \Bigl( h^2(x) +\frac{\lambda  t^2}{h^2(x)} |\nabla_x h(x)|^2 \Bigl)  \partial^2_\tau u(\tau, x)\nonumber\\
 &+ 2 \frac{\lambda  t }{h(x)}  \nabla_x h(x) \cdot \nabla_x \partial_\tau u(\t,x)+ \frac{\lambda  t }{h(x)} \Delta h(x) \partial_\t u(\t,x) \qquad  \text{for $(\t,x) \in \Omega_1$.}\label{eq:reldiffope-ND}
\end{align}
Here $\nabla_x$ and $\Delta_x$ denote the gradient and Laplacian with respect to the variable $x \in \R^m$,  and we  simply write $\nabla$ and $\Delta$ when there is no confusion.

From \eqref{eqansnew} and \eqref{eqreladiffopets} the  problem \eqref{eq:per4}  is  equivalent to 
\begin{align}\label{eq:Proe1-ss1}
  \begin{cases}
    [\widetilde{L}^h_\lambda u = 0 &\quad \textrm{ in}\quad  \Omega_1 \\
     u=0    &\quad\textrm{ on} \quad \partial  \Omega_1 \\
    h(\cdot) \sqrt{1+ \lambda \frac{ |\nabla  h(\cdot) )|^2 }{ h^4(\cdot)} }  \partial_t u (1, \cdot)=+1  &\quad\textrm{ in} \quad \R^m .\\
  \end{cases}
  \end{align}

\section{Functional setting  and  the linearized operator}\label{eq.setting}

To set up a framework for  problem   \eqref{eq:Proe1-ss1}, we  define for  $k \ge 0$ and $\a\in (0,1)$, 
$$
C^{2,\alpha}_{p,e}(\overline \Omega_1):= \{u \in C^{2,\alpha}(\overline \Omega_1)\::\: \text{$u=u (t, x)$ is  even   and $2\pi$ periodic in   $x_1, \cdots,x_m $ }\},
$$
$$
\widetilde{X}_k:= \{u \in C^{k,\alpha}_{p,e}(\overline \Omega_1)\::\: \text{$u$ is odd in $t$}\},
$$
$$
\widetilde{X}^D_k:= \{u \in \widetilde{X}_k\::\: \text{$u=0$} \quad \textrm{on}  \quad \partial \O_1 \}
$$
as well as $$
Y_k:= \{h \in C^{k,\alpha}(\R^m)\::\: \text{$h$ is even and $2\pi$ periodic in  $x_1, \cdots,x_m $ } \}. 
$$ and    $$Y_k^+:= \{h \in   Y_k   \::\: h>-1\}.$$ 

Recalling \eqref{eq:Proe1-ss1} and  \eqref{eq:normlllls},  our  aim is to prove that for  some parameter  $\lambda$, we can find the functions  $(u,h)\in  X_2 \times Y_2^+$  such that 
\begin{align}\label{eq:Proe1122}
\left \{
    \begin{aligned}
    &\widetilde{L}^{h}_{\lambda, n} u = 0 & \quad \textrm{ in}\quad  \Omega_1\\
&u=0&  \quad\textrm{ on} \quad \partial  \Omega_1 \\
&\widetilde{K}_\lambda (u, h)=0  &\quad\textrm{ in} \quad \R^m,
  \end{aligned}
       \right.
\end{align}
where  $\widetilde{K}_\lambda :  X_2 \times Y_2^+ \longrightarrow Y_1$ is defined by 
$$\widetilde{K}_\lambda (u, h):=h(\cdot) \sqrt{1+ \lambda \frac{ |\nabla  h(\cdot) )|^2 }{ h^4(\cdot)} }  \partial_t u (1, \cdot)-1.$$ 
\begin{Remark}\label{eqrresolutionform2}
Since $ (t,x) \mapsto v_n(t, x)=\frac{(-1)^n}{n\pi} \sin(n\pi t)$ solves  $ L_{\l,n} v_n=0$ in  $\R\times  \R^{m} \supset \O_{1+h}$,  we have by  \eqref{eqreladiffopets} that  
$$\widetilde{L}_{\l,n}^{1+h}  v_n(\t/(1+h))=  0  \qquad \text{in $\Omega_*$}.$$
Furthermore, 
\begin{equation}\label{eq:approxisolu2}
v_n(t/(1+h))=v_n(t, x)-t h(x) v'_n(t, x) + O(||h||^2_{C^{2,\a}(\R^m)}). 
\end{equation}
We define  $$V(t, x):=  v_n(t, x)- t h(x) v'_n(t, x)+u(t, x) $$ with $u$ and $h$ small and $u$ odd in $t$.
Then  $V(\pm 1, x) =0$ for all $x\in \R^m$ if and only if 
\be
h=h_u= u (1,\cdot)  
\ee
In addition since $v''_n(1, t)=0$, we have 
\begin{align}
\partial_t V(1, x)&= 1-u(1, x)+\partial_t u (1, x),
\end{align}
 and hence   \eqref{eq:Proe1122}  reads 
\begin{align} \label{eq:spacecondi2}
 \widetilde{K}_\lambda (V, 1+h_u)=\sqrt{1+ \lambda \frac{|\nabla h_u(x)|^2  }{ (1+h_u)^4(x)}}( 1+ h_u(x)) \Bigl(1-u(1, x)+\partial_t u (1, x)\Bigl)-1. 
\end{align}
\end{Remark}
The Remark \ref{eqrresolutionform2} allows us to  further reduce problem \eqref{eq:Proe1122}  to the finding of the only  unknown  $u$.  For that, we consider the open set  $$
\cV:= \{ u \in  \widetilde{X}_2 \::\: u(1. \cdot)  >-1 \}
$$
and define the mapping 
\begin{align}\label{eq:maptFF}
\widetilde{F}_\lambda:  \cV  \to \widetilde{X}_0  \times Y_1  , \qquad \widetilde{F}_\lambda(u):= (\widetilde{L}^{1+h_u}_{\lambda, n} ( v_n+\widetilde{M}(u)),  \widetilde{K}_\lambda (v_n+ \widetilde{M}(u), 1+h_u),
\end{align}
where  
$$ \widetilde{M}: \cV  \to \widetilde{X}_2^D, \quad  \widetilde{M}(u)(t, x): =- t h_u(x) v'_n(t, x)+u(t, x).$$
If is clear by construction that  if 
\begin{align}\label{eq to bifurcate}
\widetilde{F}_\lambda(u)=0,
\end{align}
then the function 
\begin{align}\label{eq finalsolu}
 v_n +\widetilde{M}(u)
\end{align}
solves the problem \eqref{eq:Proe1122}, with
\begin{align}\label{eq exprehhh}
h=1+u(1. \cdot).
\end{align}
It is plain that
\begin{align}\label{eqtrivialbranche}
\widetilde{F}_\lambda(0) = 0 \qquad \text{for all $\lambda >0$.}
\end{align}
Furthermore, arguing similarly as in Proposition \ref{eqlinearised}, we see that the map  $F_\lambda$  defined by \eqref{eq:maptFF} is smooth and 
\begin{align}\label{eq diffope2}
\cH_\lambda:=D \widetilde{F}_\lambda(0)(v)=  \Bigl( \pi^2 n^2 v+ \lambda \Delta_{x}v + v_{tt}, v_t(1, \cdot)  \Bigl).
\end{align}

\section{Study of the linearised  operator  $\cH_\lambda$}\label{exkus}
In this section,  we study the operator $\cH_\lambda$ and determined its kernel. Let us first study the ODE
\begin{equation}\label{eq:ODE4}
\begin{cases}
\displaystyle b''+ \mu b=0,\quad \textrm{in}\quad (0,1) \\
\displaystyle \displaystyle b(0)=0,\quad  b'(1)=0.
\end{cases}
\end{equation}
where $\mu:=\pi^2 n^2- \l  k^2$. 
Then we distinguish the following cases.
\subsection*{The case $\mu =0$}
It is clear that the solution  of  \eqref{eq:ODE4} is trivial in this case. 
\subsection*{The case $\mu< 0$.}
A fundamental system of the linear equation is then given by
$$
\phi_1^\mu(t)= e^{-\sqrt{-\mu }t}, \qquad  \phi_2^\mu(t)= e^{\sqrt{-\mu}t}
$$
and  \eqref{eq:ODE4}  as no solution in this case. 
\subsection*{The case $0< \mu $.}
In this case, a fundamental system of the linear equation is given by
$$
\phi_1^\mu(t)= \cos (\sqrt{\mu}t), \qquad  \phi_2^\mu(t)= \sin (\sqrt{\mu}t)
$$
and we have 
\begin{align}\label{solODE}
b(t) =A \sin (\sqrt{\mu}t)
\end{align}
for some  real constant $A$.
Consequently  for $A\ne 0$,
$$
b'(1)=0 
$$
iff and only if  
\begin{equation}\label{seqofzeroewws} 
\sqrt{\mu}= (\frac{1}{2}+\ell)\pi, \quad \ell \in \N.
\end{equation}
In the following, we let 
$$
H^{k}_{p, odd}(  \Omega_1):= \{ u \in H^k(\O_1)\::\: \text{$u$ is odd  in ${\t}$, $2\pi$ periodic and  even   $x_1, \cdots,x_m$ } \},
$$
$$
\cH^{k}_{p, odd}(  \Omega_1):= \{ u \in H^k_{p, odd}(\O_1)\::\: \text{$\partial_t u(1, \cdot)=0$ on $ \R^m$  } \}
$$
 and 
$$
\cH^{k}_{\mathcal{P} }(  \Omega_1):= \{ u \in \cH^{k}_{p, odd}(  \Omega_1),\quad  u(\cdot , x)=  u(\cdot, \bold{p}(x))  \text{ for all $ x \in \R^m$, $\bold{p} \in \mathcal{P} $}\}. 
$$
We  also put  
\begin{align}\label{eq:foroofc21}
\gamma_n:=\pi^2 n^2-\frac{\pi^2}{4}  \quad \textrm{and}\quad  J_\ell:=(\frac{1}{2}+\ell)\pi
\end{align}
 and 
\begin{align}\label{eq diffope2}
\widetilde{L}_{\lambda, n}:= \pi^2 n^2  \textrm{id}+ \lambda \Delta_{x}v + v_{tt}.
\end{align}
\begin{Lemma}\label{eq:lemmkerimH}
The  set of non-trivial solutions  $v \in  \cH^{2}_{\mathcal{P} }$ to $ \cH_{\gamma_n}v=0$ in $\O_1$  is  spanned  by  
\begin{align}\label{eq elkernel}
\widetilde{v}_0(t, x)= \sin(\frac{\pi t}{2} ) (\cos(x_1)+\cdots+\cos(x_m)). 
\end{align}
Moreover, 
\be\label{eq:inclusion3}
Im \cH_{\gamma_n} = E_{0}^{\perp},
\ee
where 
\be\label{cokernel2}
E_{0}^{\perp} := \left\lbrace (w, h)\in \widetilde {X}_0\times Y_1: \int_{ \Omega_1}   w (t, x)   \widetilde{v}_0(t, x) \,dxdt- \int_{\partial  \Omega_1 } h(x) \widetilde{v}_0(t, x) dt dx=0 \right\rbrace.
\ee
\end{Lemma}

\begin{proof}
We write any function $ v \in  \cH^{2}_{\mathcal{P} }$  as 
\be\label{decomposi22}
v(t,x)= \sum_{k \in \N \cup
\{0\} }  v_k(t) \omega_{k}(x), \quad \textrm{with} \quad 
\omega_{k}(x)=\sum_{j=1}^m \cos(k x_j). 
\ee
 Then   $ v \in \cH^{2}_{p, odd}(  \Omega_1)$ is solution  to   $ \cH_{\gamma_m}v=0$ in $\O_1$  if and only if the coefficients  $v_k(t)$ satisfy the ODE \eqref{eq:ODE4}. It follows from  \eqref{seqofzeroewws} and \eqref{solODE} that  $v_k(t)$ is non-trivial if and only if 
$$
\sqrt{\pi^2 n^2- \gamma_n |k|^2}= J_\ell \Longleftrightarrow k^2 =\frac{\pi^2 n^2-J^2_\ell}{\pi^2 n^2-J^2_0} \leq 1 \Longleftrightarrow k=1 \quad \textrm{and}\quad \ell=0. 
$$
 and  $v_1=\sin(\frac{\pi t}{2})$ yielding  \eqref{eq elkernel}. 

To prove  \eqref{eq:inclusion3}, we consider the  family  $\widetilde{\omega}_k:=\frac{\omega_{k}}{\|\omega_k\|_{L^2}}$  and similarly as  in  \eqref{eqdecomp00}, we can write 
$$
v(t,x)= \sum_{k \in \N \cup \{0\}, \ell \in \N}  v_{k,\ell} \widetilde{ \psi}_{\ell }(t)\widetilde{\omega}_k (x),
$$
where  $\widetilde{ \psi_\ell}=\frac{\psi_{\ell}}{\|\psi_{\ell}\|_{L^2(-1, 1)}}$, $\psi_{\ell}(t)=\sin(J_\ell t)$ and $J_\ell$ is defined in  \eqref{eq:foroofc21}. Then the norm 
$\|u\|_{H^1_{p, odd }(\O_1)}$ is equivalent to $ \sum_{k, \ell } (1+\ell^2+k^2)u_{k,\ell}^2$.
We also note that the corresponding bilinear form  $\widetilde{\cB}$  to $\cH_\lambda$ satisfies 
\begin{align*}
\widetilde{\cB}(u,u)=  \sum_{ k, \ell  } \widetilde{\s}_{k,\ell }v_{k,\ell}^2,  \quad \textrm{with}\quad \widetilde{\s}_{k,\ell }:= J_\ell-\pi^2n^2+k^2\gamma_n. 
\end{align*}
Hence, it follows from   \eqref{eq:foroofc21} that $  \widetilde{\s}_{k,\ell } \backsim \ell ^2+k^2$ as $\ell^2+k^2 \rightarrow \infty$.  
The proof of \eqref{eq:inclusion3} therefore  follows  step by step  Lemma \ref{eq:lemmkerimG}(iii). 
\QED
\end{proof}

\section{Proof of Theorem \ref{Theo2-DN}}\label{sec:compl-proof-theor}

The proof of Theorem~\ref{Theo2-DN} will be completed by applying the
Crandall-Rabinowitz Bifurcation theorem \cite{M.CR}. To proceed, we  consider the spaces 
\begin{align*}
&X:= \{ u \in  \widetilde{X}_2\::\: u(\cdot , x)=  u(\cdot, \bold{p}(x))  \text{ for all $ x \in \R^m$, $\bold{p} \in \mathcal{P} $}\},\\
&Y:= \{ (u, h)\in \widetilde{X}_0 \times Y_1 \::\:  u(\cdot , x)=  u(\cdot, \bold{p}(x)), \quad h(x)=  h(\bold{q}(x))  \text{ for all $ x \in \R^m$ $  \bold{p}, \bold{q} \in
\mathcal{P} $}\}.
\end{align*}
It is plain from \eqref{eq:derivepermutrule}  that  $\widetilde{F}_\lambda$  in  \eqref{eq:maptFF} maps $\cV\cap X$ into $Y$. We consider the open set
\begin{equation}
  \label{eq:def-cO4}
\cO:= \{(\lambda, u) \in \R \times X \::\: \lambda>0, \: u(1, \cdot)  > -1 \}  \subset \R \times X
\end{equation}
and define the operator \be \label{eq:def--G4}  S: \cO \subset \R \times X  \to Y,
\qquad S(\lambda,  u) =\widetilde{F}_\l(u). \ee
Then from  \eqref{eqtrivialbranche}, we have 
$$S(\lambda,  0) = 0 \qquad \text{for all $\lambda >0$.}$$ and moreover,
\begin{equation}
  \label{eq:cGH-lambda}
 D_ u  S (\lambda,0)=  D_ u \widetilde{F}_\lambda(0)\big|_X = \cH_\lambda|_X \in \cL(X,Y).
\end{equation}
We have the following.

\begin{Proposition}\label{propPhi4}
The linear operator
$$
\cH_n:=\cH_{\gamma_n}\big|_X \in \cL(X,Y )
$$
has the following properties.
\begin{itemize}
\item[(i)] The kernel $N(\cH_n)$ of $\cH_n$ is spanned by the function
  \begin{equation}
    \label{eq:def-v-04}
\widetilde{v}_0 \in X, \qquad \widetilde{v}_0(t, x)= \sin(\frac{\pi t}{2} ) (\cos(x_1)+\cdots+\cos(x_m))
  \end{equation}
\item[(ii)] The range of $\cH_n$ is given by
$$
R(\cH_n)= E_{0}^{\perp}
$$
\end{itemize}
Moreover,
\begin{equation}
  \label{eq:transversality-cond4}
\partial_\lambda \Bigl|_{\lambda= \gamma_n} \cH_\lambda \widetilde{v}_0 \;\not  \in \; R(\cH_n).
\end{equation}
\end{Proposition}
\begin{proof}
$ (i)$   and (ii)  obviously follow from Lemma \ref{eq:lemmkerimH}. Finally, using \eqref{eq diffope2}, we  find 
$$
\partial_\lambda \Bigl|_{\lambda= \gamma_n} \cH_\lambda \widetilde{v}_0  = (\Delta_{x}\widetilde{v}_0, 0)=(- \widetilde{v}_0 , 0).
$$
The  proof is complete.
   \QED
\end{proof}
The following result provides the ingredients needed to complete the proof of Theorem \ref{Theo2-DN}. 
\begin{Theorem}\label{Theo1-general-ND22}
For every $n\in \N$, there  exist  ${\rho_n}>0$ and a smooth curve
$$
(-{\rho_n},{\rho_n}) \to   (0,+\infty) \times X,\qquad s \mapsto (\gamma_n(s),\phi^n_s)
$$
with
$\gamma_n(0)= \gamma_{n}$, $\phi_n(0)\equiv 0$ such that  
\be
G_{\gamma_n(s)}(\phi_n(s)) =0.
\ee
Moreover, $\phi^n_s = s(\widetilde{v}_{0}+\chi_n(s)),$ with a smooth curve 
$$
(-{\rho_n},{\rho_n}) \to X, \qquad s \mapsto  \chi_n(s)
$$ satisfying  $ \o_n(0) =0$ and 
\begin{align*}
\int_{\O_1}   \chi_n(s) (\t,x)  \widetilde{v}_{0}(\t,x)\,dxd\t=0,
\end{align*}
where
$$\widetilde{v}_0(t, x)= \sin(\frac{\pi t}{2} ) (\cos(x_1)+\cdots+\cos(x_m)).$$  
In addition, setting
 \begin{equation}\label{solutinter2}
\widetilde{u}_s(\t, x):=  \phi^n_s(\t,x) -  (-1)^n \cos( n\pi t)\phi^n_s(1, x)  + v_n(\t, x)
 \end{equation}
and  
 \begin{equation}\label{eexpreh2}
 h _{\varphi^n_s}(x):=\phi^n_s(1, x), 
 \end{equation}
 the function
\begin{align}\label{eqans2-ND2}
V_s(t,x)=\widetilde{u}_s((1+h _{\phi^n_s})t,x)
\end{align} satisfies 
\begin{equation}\label{eq:solved-ND2}
  \left \{
    \begin{aligned}
\mathcal{L}^n_s V_s&=0&& \qquad \text{in $ \Omega_{  {1+h _{\phi^n_s}} }$,}\\
             V_s&=0 &&\qquad \text{on $\partial \Omega_{ 1+h_{\phi^n_s}}$,}\\
  \frac{\partial V_s}{\partial \eta_s}  &= \pm1   &&\qquad \text{on $\partial  \Omega^{\pm}_{ 1+ h_{\phi^n_s}}$} \end{aligned}
       \right.
\end{equation}
where 
$$ \mathcal{L}^n_s:=\gamma_n(s) \Delta_{x}  +\partial_{tt}  +  n^2\pi^2  \textrm{id}$$
 and $\eta_s$ denotes the outer  unit vector filed of the boundary  $\partial  \Omega_{ 1+ h_{\phi^n_s}}$. 
\end{Theorem}

\begin{proof}
 We consider the smooth map $ S: \cO \subset \R \times X  \to Y$ defined in \eqref{eq:def--G4}  and  set 
\begin{align}\label{eq:spaorthokernel}
X^{\perp} := \left\lbrace  v \in X: \int_{\O_1 } v(t,x)\widetilde{v}_{0}(t,x)\,dxdt =0\right\rbrace.
\end{align}
By Proposition \ref{propPhi4} and  the Crandall-Rabinowitz Theorem (see \cite[Theorem 1.7]{M.CR}), we then find ${\rho_n}>0$ and a smooth curve
$$
(-{\rho_n },{\rho_n }) \to  \mathcal{O} \subset  \R_+ \times X_2 , \qquad s \mapsto (\gamma_n(s), \phi^n_s)
$$
such that
\begin{enumerate}
\item[(i)] $S(\gamma_n(s),  \phi^n_s)=0$ for $s \in (-{\rho_n },{\rho_n })$,
\item[(ii)] $ \gamma_n (0)= \gamma_{n}$, and
\item[(iii)]  $\phi^n_s = s \widetilde{v}_{0}+ s \chi_n(s) $ for $s \in (-{\rho_n },{\rho_n })$ with a smooth curve
$$
(-{\rho_n },{\rho_n }) \to X^{\perp}, \qquad s \mapsto \chi_n(s)
$$
satisfying $ \chi_n(0) =0$
and
$$\int_{\O_1 } \chi_n(s) (t,x) \widetilde{v}_{0}(t,x)\,dxdt=0.$$
\end{enumerate}
Since $S(\gamma_n(s),  \phi^n_s)=0$ for $s \in (-{\rho_n },{\rho_n })$, recalling  \eqref{eq:def--G4}, \eqref{eq to bifurcate},  \eqref{eq exprehhh}  and  \eqref{eq exprehhh}, the  function $\widetilde{u}_s$ in \eqref{solutinter2} solves   \eqref{eq:Proe1122}, with $h$ given by  \eqref{eexpreh2}. Therefore  by  \eqref{eqansnew}, $V_s$ in \eqref{eqans2-ND2} is a solution of  \eqref{eq:solved-ND2}.  
\QED
\end{proof}

\subsection*{Proof of Theorem \ref{Theo2-DN} (completed)}
By  Theorem  \ref{Theo1-general-ND22}  and \eqref{eq:changsoluD2}, the function
\begin{equation}\label{Solfinalw}
(t,x )\mapsto \widetilde{w}_s(t,x)=  V_s( t/ \sqrt{\gamma_n(s)},x)= \widetilde{u}_s((1+h _{\phi^n_s})/\sqrt{\gamma_n(s)}t,x)
 \end{equation}
 solves \eqref{eq:solved-main-ND} with $\mu^n_s=\frac{n^2\pi^2}{\gamma_n(s)} $ and 
 $$
 h^n_s(x)=\frac{1+h_{\phi^n_s}(x)}{\sqrt{\gamma_n(s)}}.  
 $$
From Theorem   \ref{Theo1-general-ND22}
$$\phi^n_s(t, x) = s \widetilde{v}_0(t, x) + o(s),$$ where $o(s) \to 0$ in $C^2$-sense in $\ov{\Omega_1}$ as $s \to 0$, with 
$$\widetilde{v}_0(t, x)= \sin(\frac{\pi t}{2} ) (\cos(x_1)+\cdots+\cos(x_m)).$$  
Hence using \eqref{eexpreh2}, 
\begin{equation}\label{eexpre2}
 h _{\varphi^n_s}(x):= s\vartheta(x)+ o(s) \quad \textrm{with}\quad \vartheta(x):=\cos(x_1)+\cdots+\cos(x_m)\nonumber\\, 
 \end{equation}
and therefore, 
\begin{align*}
  h^n_{s}(x)&=\frac{1+h_{\phi^n_s}}{\sqrt{ \gamma_n(s)}}=  \frac{1}{\sqrt{\gamma_n(s)}}+ \frac{s}{\sqrt{\gamma_n(s)}}\vartheta(x)+ o(s)\\
&=  \frac{1}{n\pi} \sqrt{\mu^n_s} + s\frac{1}{\sqrt{n^2\pi^2-\frac{\pi^2}{4}}}\vartheta(x) + o(s) \qquad \text{as $s \to 0$},
\end{align*}
where we have used $\gamma_n(0)=\gamma_n=\pi^2 n^2-\frac{\pi^2}{4}$  from \eqref{eq:foroofc21}.

In addition, \eqref{solutinter2} yields
\begin{align}
&\widetilde{w}_s(\frac{t}{h^n_s(x))},x)= \widetilde{u}_s(\t, x):=  \phi^n_s(\t,x) -  (-1)^n\cos( n\pi t)\phi^n_s(1, x)  + v_n(\t) \nonumber\\ 
  &=v_n(t) + s \Bigl(\sin(\frac{\pi t}{2} )- (-1)^n \cos( n\pi t \Bigr)\vartheta(x)  + o(s),\nonumber
\end{align}
where $o(s) \to 0$ in $C^1$-sense on $\Omega_*$.
Thus, the constants in  Theorem \ref{Theo2-DN} are given by 
\begin{align*}
d_n& = \frac{n^2}{n^2- \frac{1}{4}},\qquad   a_n = \frac{1}{n\pi},\qquad  b_n =\frac{1}{\sqrt{n^2\pi^2- \frac{\pi^2}{4}}}.
\end{align*}
\QED

\section{Crandall-Rabinowitz bifurcation theorem} \label{eq: Cradal Rabi1}

\begin{Theorem}[Crandall-Rabinowitz bifurcation theorem, \cite{M.CR}] \label{eq: Cradal Rabi}
Let $X$  and $Y$ be two Banach spaces, $U\subset X$ an open set of
$X$ and $I$ an open interval  of $\R$. We assume that  $0\in U$.
Denote by  $\varphi$ the elements of $U$  and  $\lambda$ the
elements of  $I$. Let $F: I\times U\rightarrow Y$  be a twice
continuously differentiable function such that
 \begin{enumerate}
\item
$ F(\lambda, 0)=0\quad \textrm{ for all }\quad \lambda \in I,$
 \item
 $\ker(D_\varphi F(\lambda_*, 0))=\R \varphi_*$  for some $\lambda_*\in I$  and $\varphi_*\in X\setminus \{0\}$,
 \item
 $\textrm{Codim Im}(D_\varphi F(\lambda_*, 0))=1,$
 \item
 $D_\lambda D_\varphi F(\lambda_*, 0)(\varphi_*) \notin \textrm{Im}(D_\varphi F(\lambda_*, 0)$.
 \end{enumerate}
 Then for any complement $Z$ of the subspace $\R \varphi_*$,  spanned by  $\varphi_*$, there exists a continuous curve
   $$(-\e, \e)\longrightarrow \R\times Z, \quad s\mapsto (\lambda(s), \varphi(s))$$  such that
 \begin{enumerate}
 \item
 $\lambda(0)=\lambda_*, \quad \varphi(0)=0,$
 \item
 $s(\varphi_*+\varphi (s))\in U,$
 \item
 $F(\lambda(s), s(\varphi_*+\varphi (s))=0$.
 \end{enumerate}
 Moreover, the set of solutions to the equation  $F(\lambda, u)=0$ in a neighborhood of  $(\lambda_*, 0)$
 is given by  the curve $\{ (\lambda, 0), \lambda\in \R\}$  and  $\{ s(\varphi_*+\varphi (s)), s\in(-\e, \e)\}$.
\end{Theorem}

\end{document}